\documentclass[12pt]{amsart}
\usepackage{latexsym,amsxtra,amscd,ifthen,amsmath,color, multicol,hyperref,mathdots,mathrsfs}
\usepackage{amsfonts}
\usepackage{verbatim}
\usepackage{amsmath}
\usepackage{amsthm}
\usepackage{amssymb}
\usepackage[all,cmtip]{xy}
\usepackage{changes}
\usepackage{enumerate}
\usepackage[latin1]{inputenc}
\usepackage{amsmath}
\usepackage{amsfonts}
\usepackage{amssymb}
\usepackage{graphicx,mathtools}
\usepackage{multirow}
\usepackage{latexsym}
\usepackage{amssymb}
\usepackage{placeins} 
\allowdisplaybreaks

\usepackage{graphicx}

\usepackage{tikz}
\usetikzlibrary{mindmap,shapes,snakes,matrix}
\usepackage{cleveref} 

\theoremstyle{plain}
\newtheorem{thm}{Theorem}[subsection]
\newtheorem{lem}[thm]{Lemma}
\newtheorem{prop}[thm]{Proposition}
\newtheorem{cor}[thm]{Corollary}

\theoremstyle{definition}

\newtheorem{remark}[thm]{Remark}

\def\k{\ensuremath{\bold{k}}}

\def\J{\mathcal{J}}
\def\D{\mathcal{D}}
\def\B{\mathcal{B}}
\def\O{\mathcal{O}}

\def\q{\mathfrak{q}}

\newcommand*{\Hom}{\ensuremath{\text{\upshape Hom}}}

\newcommand*{\gr}{\ensuremath{\text{\upshape gr}}}
\newcommand*{\grH}{\ensuremath{\text{\upshape gr}}\, H}

\newcommand*{\Ker}{\ensuremath{\text{\upshape Ker}}}
\newcommand*{\Img}{\ensuremath{\text{\upshape Im}}}

\newcommand*{\ad}{\ensuremath{\text{\upshape ad}}}
\newcommand*{\Prim}{\ensuremath{\text{\upshape P}}}

\newcommand*{\id}{\ensuremath{\text{\upshape id}}}
\newcommand*{\YD}{\,^{G}_{G} \mathcal{YD}}
\newcommand*{\BV}{\mathcal{B}(V)}

\newcommand*{\HL}{\ensuremath{\text{\upshape H}}}
\newcommand*{\HH}{\ensuremath{\text{\upshape HH}}}

\def\dim{\operatorname{dim}}

\newcommand*{\Bock}{\boldsymbol{\omega}}

\DeclarePairedDelimiterX\set[1]\lbrace\rbrace{#1}

\begin{document}
\numberwithin{equation}{thm}
\thispagestyle{empty}

\title[Pointed $p^3$-dimensional Hopf algebras in positive characteristic]{Pointed $p^3$-dimensional Hopf algebras \\ in positive characteristic}

\author{Van C. Nguyen}
\author{Xingting Wang}

\address{Department of Mathematics\\
Northeastern University\\
Boston, MA 02115}
\email{v.nguyen@northeastern.edu}

\address{Department of Mathematics\\
Temple University\\
Philadelphia, PA,  19122-6094}
\email{xingting@temple.edu}




\begin{abstract}
We classify pointed $p^3$-dimensional Hopf algebras $H$ over any algebraically closed field $\k$ of prime characteristic $p>0$. In particular, we focus on the cases when the group $G(H)$ of group-like elements is of order $p$ or $p^2$, that is, when $H$ is pointed but is not connected nor a group algebra. This work provides many new examples of (parametrized) non-commutative and non-cocommutative finite-dimensional Hopf algebras in positive characteristic. 
\end{abstract}

\maketitle


\section*{Introduction}

The classification of pointed $p^3$-dimensional Hopf algebras $H$ in characteristic zero (e.g.~in $\mathbb{C}$) was independently presented by Andruskiewitsch and Schneider \cite{NA98, NA02, NA10}, by Caenepeel and D\u{a}sc\u{a}lescu \cite{CaenepeelDascalescu}, and by Stefan and van Oystaeyen \cite{SteVan} using different methods. In this paper, we provide the classification in characteristic $p>0$. This work, together with the classification of connected $p^3$-dimensional Hopf algebras \cite{NWW1, NWW2} and that of groups of order $p^3$, provides isomorphism classes of pointed Hopf algebras of dimension $p^3$ over an algebraically closed field of characteristic $p$. 

Moreover, we remark that the classification of $p^n$-dimensional pointed Hopf algebras over an algebraically closed field of prime characteristic $q>0$, where $p$ and $q$ are coprime, yields similar isomorphism classes as in the case of characteristic zero; since the same technique should work in both cases. Hence, our work in this paper in characteristic $p>0$, combined with previous classification results by other authors, will compete the classification picture for pointed Hopf algebras of dimensions $p, p^2$, and $p^3$ over any algebraically closed field of \textit{arbitrary} characteristic. 

We present the following diagram outlining our classifying idea in this paper. We first break down to smaller cases by the order of the group $G(H)$ of grouplike elements. We observe that in the case when $|G(H)|=p$, it occurs that the braided Hopf algebra $R$ in the associated graded Hopf algebra $\grH \cong R \# \k G(H)$ of $H$, cf.~Section~\ref{sec:prelim}, may not be primitively generated, and we could have Yetter-Drinfeld modules in $\YD$ of either diagonal type or Joran type; hence, further cases occur where the structures arise. Interested readers may refer to corresponding section(s) for detailed classification results.  

At last, we emphasize that the principle proposed by Andruskiewitsch and Schneider in \cite{NA98} to study pointed Hopf algebras in characteristic zero is generally applicable in positive characteristic. But the difficulty arises when the characteristic of the base field $\k$ divides the dimension of the pointed Hopf algebra $H$, because in this situation: (1) the braided Hopf algebras $R$ are in general not primitively generated (that is, they are not Nichols algebras) even when $G(H)$ is abelian, and (2) the liftings from $\grH$ to $H$ are computationally challenging in characteristic $p$. \\

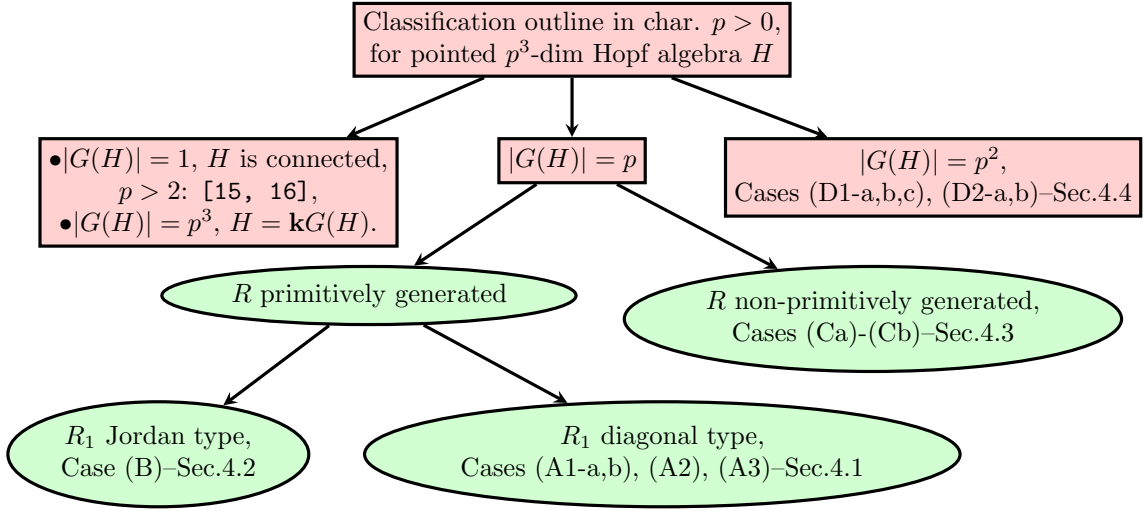
\FloatBarrier
{\Small
\begin{tikzpicture}
\tikzstyle{iellipse}=[draw=black,shape=ellipse,very thick,fill=green!18!white];  
\tikzstyle{ibox}=[draw=black,shape=rectangle,very thick,fill=red!18];
\node[ibox, align=center, below] (n1) at (-2,0)  {Classification outline in char. $p>0$, \\ for pointed $p^3$-dim Hopf algebra $H$}; 
\node[ibox, align=center, below] (n2) at (-6.7,-1.8)  {$\bullet |G(H)|=1$, $H$ is connected, \\$p>2$: \texttt{\cite{NWW1, NWW2}}, \\ $\bullet |G(H)|=p^3$, $H=\k G(H)$.}; 
\node[ibox, align=center, below] (n3) at (-2,-1.8)  {$|G(H)|=p$};
\node[ibox, align=center, below] (n4) at (2.8,-1.8)  {$|G(H)|=p^2$, \\ Cases (D1-a,b,c), (D2-a,b)--Sec.\ref{liftingD}};
\node[iellipse, align=center, below] (n5) at (-4.7,-3.5)  {$R$ primitively generated};
\node[iellipse, align=center, below] (n6) at (2,-3.5)  {$R$ non-primitively generated, \\ Cases (Ca)-(Cb)--Sec.\ref{liftingC}};
\node[iellipse, align=center, below] (n7) at (-7.5,-5.3)  {$R_1$ Jordan type, \\ Case (B)--Sec.\ref{liftingB}};
\node[iellipse, align=center, below] (n8) at (-0.8,-5.3)  {$R_1$ diagonal type, \\ Cases (A1-a,b), (A2), (A3)--Sec.\ref{liftingA}};

\draw[->, very thick, >=stealth] (n1) -- (n2);
\draw[->, very thick, >=stealth] (n1) -- (n3);
\draw[->, very thick, >=stealth] (n1) -- (n4);
\draw[->, very thick, >=stealth] (n3) -- (n5);
\draw[->, very thick, >=stealth] (n3) -- (n6);
\draw[->, very thick, >=stealth] (n5) -- (n7);
\draw[->, very thick, >=stealth] (n5) -- (n8);
\end{tikzpicture}
}

We obtain the following classes of pointed $p^3$-dimensional Hopf algebras in characteristic $p>0$: \\

{\small 
\begin{tikzpicture}
\tikzstyle{ibox}=[draw=black,shape=rectangle,very thick,fill=white];
\node[ibox, align=center, below] (n1) at (0,0)  {\textbf{Case A1.} Liftings from $\grH=\k \langle a, b, g \rangle /(g^p = 1, a^p=b^p=0, ab=ba, \,ga = ag, \,gb=bg)$, \\ with $\Delta(g)= g \otimes g, \, \Delta(a)=a \otimes 1+g \otimes a, \, \Delta(b)=b \otimes 1+g^u\otimes b$, \\ $\varepsilon(g)=1, \, \varepsilon(a)=\varepsilon(b)=0, \, S(g)=g^{-1}, \, S(a) = -ag^{-1}, \, S(b)= -bg^{-u}$, for $0 \le u\le p-1$.};
\end{tikzpicture}}

\vspace{0.5em}
When $u=0$, there are $1$ infinite parametric family and $10$ finite classes of $H$ having structured lifted from case (A1). When $u \neq 0$, there are $2(p-1)$ infinite parametric families and $6(p-1)$ finite classes of $H$. \\

{\small 
\begin{tikzpicture}
\tikzstyle{ibox}=[draw=black,shape=rectangle,very thick,fill=white];
\node[ibox, align=center, below] (n1) at (0,0)  {\textbf{Case A2.} Liftings from $\grH=\k \langle a, b, g \rangle /(g^p = 1, \, a^p=b^p=0, \, ab=ba, \,ga = ag, \,gb=bg)$, \\ with $\Delta(g)= g \otimes g, \, \Delta(a)= a \otimes 1+1 \otimes a, \, \Delta(b)=b \otimes 1+1 \otimes b$, \\ $\varepsilon(g)=1, \, \varepsilon(a)=\varepsilon(b)=0, \, S(g)=g^{-1}, \, S(a) = -a, \, S(b)= -b$.};
\end{tikzpicture}}

\vspace{0.5em}
There are $5$ finite classes of $H$ having structured lifted from case (A2). \\

{\small 
\begin{tikzpicture}
\tikzstyle{ibox}=[draw=black,shape=rectangle,very thick,fill=white];
\node[ibox, align=center, below] (n1) at (0,0)  {\textbf{Case A3.} $(p=2)$ Liftings from $\grH=\k \langle a, b, g \rangle /(g^p = 1, \, a^p=b^p=0, \, ab=ba$, \\ $ga = bg, \,gb=ag)$, with $\Delta(g)= g \otimes g, \, \Delta(a)= a \otimes 1+1 \otimes a, \, \Delta(b)=b \otimes 1+1 \otimes b$, \\ $\varepsilon(g)=1, \, \varepsilon(a)=\varepsilon(b)=0, \, S(g)=g^{-1}, \, S(a) = -a, \, S(b)= -b$.};
\end{tikzpicture}}

\vspace{0.5em}
There are $5$ finite classes of $H$ having structured lifted from case (A3). \\

{\small 
\begin{tikzpicture}
\tikzstyle{ibox}=[draw=black,shape=rectangle,very thick,fill=white];
\node[ibox, align=center, below] (n1) at (0,0)  {\textbf{Case B.} $(p > 2)$ Liftings from $\grH=\k \langle a, b, g \rangle /(g^p = 1,\, a^p=b^p=0,\, ab-ba=\frac{1}{2}a^2$, \\ $ga=ag, \, gb=(a+b)g)$, with $\Delta(g)= g \otimes g, \, \Delta(a)=a \otimes 1+g\otimes a, \, \Delta(b)=b \otimes 1+g\otimes b$, \\ $\varepsilon(g)=1, \, \varepsilon(a)=\varepsilon(b)=0, \, S(g)=g^{-1}, \, S(a) = -ag^{-1}, \, S(b)=(a-b)g^{-1}$.};
\end{tikzpicture}}

\vspace{0.5em}
Due to complicated computations in characteristic $p$, the lifting in case (B) is not clear in general; however, we show in \Cref{liftingB} the case when $p=3$ for illustration and make a conjecture for the lifting for $p>3$. \\

{\small 
\begin{tikzpicture}
\tikzstyle{ibox}=[draw=black,shape=rectangle,very thick,fill=white];
\node[ibox, align=center, below] (n1) at (0,0)  {\textbf{Case C.} Liftings from $\grH=\k \langle a, b, g \rangle /(g^p = 1,\, a^p=b^p=0,\, ab=ba, \, ga = ag$, \\ $gb=bg)$, with $\Delta(g)= g \otimes g,\, \Delta(a)=a \otimes 1+g^\epsilon \otimes a$, \\ $\Delta(b)=b \otimes 1 +1 \otimes b + \sum_{1 \le i \le p-1} \frac{(p-1)!}{i!\,(p-i)!}\, (a^i \, g^{\epsilon(p-i)} \otimes a^{p-i} )$, \\ $\varepsilon(g)=1, \, \varepsilon(a)=\varepsilon(b)=0, \, S(g)=g^{-1}, \, S(a) = -ag^{-\epsilon}, \, S(b)= -b$, for $\epsilon \in \{0,1\}$.};
\end{tikzpicture}}

\vspace{0.5em}
When $\epsilon=0$, we get $1$ infinite parametric family and $2$ finite classes of $H$. When $\epsilon=1$, due to complicated computations in characteristic $p$, the lifting in this case (Cb) is not clear in general. We show in \Cref{liftingC} the cases (C) when $p=2$, and when $p>2$ with additional assumption $gx=xg$, where $x$ is the lifting of $a$ from $\gr H$ to $H$. \\

{\small 
\begin{tikzpicture}
\tikzstyle{ibox}=[draw=black,shape=rectangle,very thick,fill=white];
\node[ibox, align=center, below] (n1) at (0,0)  {\textbf{Case D1.} Liftings from $\grH=\k \langle a, g \rangle /(g^{p^2} = 1, \,a^p=0, \, ga = ag)$, \\ with $\Delta(g)= g \otimes g, \, \Delta(a)=a \otimes 1+g^{\epsilon} \otimes a$, \\ $\varepsilon(g)=1, \, \varepsilon(a)=0, \, S(g)=g^{-1}, \, S(a) = -ag^{-\epsilon}$, for $\epsilon \in \{0,1,p\}$.};
\end{tikzpicture}}

\vspace{0.5em}
When $\epsilon=0$, we get $2$ finite classes of $H$. When $\epsilon=1$, there are $1$ infinite parametric family and $2$ finite classes of $H$. When $\epsilon=p$, there are $2$ finite classes of $H$. \\

{\small 
\begin{tikzpicture}
\tikzstyle{ibox}=[draw=black,shape=rectangle,very thick,fill=white];
\node[ibox, align=center, below] (n1) at (0,0)  {\textbf{Case D2.} Liftings from $\grH=\k \langle a, g_1, g_2 \rangle /(g_1^{p} = g_2^{p}=1, \, a^p=0, \,g_ia = ag_i)$, \\ with $\Delta(g_i)= g_i \otimes g_i, \, \Delta(a)=a \otimes 1+g_1^\epsilon \otimes a$, \\ $\varepsilon(g_i)=1, \, \varepsilon(a)=0, \, S(g_i)=g_i^{-1}, \, S(a) = -ag_1^{-\epsilon}$, for $\epsilon \in \{0,1\}$.};
\end{tikzpicture}}

\vspace{0.5em}
When $\epsilon=0$, we have $2$ finite classes of $H$. When $\epsilon=1$, there are $1$ infinite parametric family and $2$ finite classes of $H$. \\

The paper is organized as follows. In Section~\ref{sec:prelim}, we recall some basic notations and properties of pointed Hopf algebras, Yetter-Drinfeld modules and Nichols algebras. In Section~\ref{sec:BVrank2}, we study rank two Nichols algebras over $\k G$ where $\k$ is of arbitrary characteristic and $G$ is a finite group. In Section~\ref{sec:R}, we study the bosonizations of these braided Hopf algebras under the assumption that $\k$ is of characteristic $p$ and $G$ is a $p$-group of order $\le p^2$. In Section~\ref{sec:pointed p3}, we apply the Lifting Method to these bosonizations to obtain our classification results. Finally, we point out that rank two Nichols algebras of diagonal type over fields of positive characteristic were studied in \cite{WHeck}; and the Nichols algebras in Case (B) were first discussed in \cite{CLW} as examples of Nichols algebras in positive characteristic.  


\section{Preliminary}
\label{sec:prelim}

We first recall some general results over a base field $\k$ of arbitrary characteristic. The unadorned tensor $\otimes$ means $\otimes_\k$ unless specified otherwise. Let $H:=H(m,u,\Delta,\varepsilon,S)$ be any finite-dimensional pointed Hopf algebra with its \textit{coradical} $H_0$ (the sum of all simple subcoalgebras of $H$) a group algebra $H_0 = \k G$; that is, $H_0$ is a Hopf subalgebra of $H$ generated by the grouplike elements $G:=G(H)=\{ g \in H \,|\, \Delta(g)=g \otimes g\}$. Note that the dimension of $H_0$ (and hence the order of group $G$) must divide the dimension of $H$ by Nichols-Zoeller's freeness theorem \cite{NZ}.

Let 
$$H_0 \subseteq H_1 \subseteq H_2  \subseteq \cdots \subseteq H $$ 
be the \textit{coradical filtration} of $H$, where $H_n=\Delta^{-1}(H\otimes H_{n-1} + H_0\otimes H)$ inductively, see \cite[Chapter 5]{MO93}. For pointed Hopf algebras $H$, this is indeed a Hopf algebra filtration since the coradical $H_0= \k G$ is a Hopf subalgebra of $H$ \cite[Lemma 5.2.8]{MO93}. Hence, we can consider the associated graded Hopf algebra $\grH = \bigoplus_{n \geq 0} H_n/H_{n-1}$, with convention $H_{-1}=0$. Note that the zero term of $\grH$ equals its coradical, i.e. $(\grH)_0=H_0$. There is a projection $\pi: \grH \rightarrow H_0$ and an inclusion $\iota: H_0 \rightarrow \grH$ such that $\pi \iota = \id_{H_0}$. Let $R$ be the algebra of coinvariants of $\pi$:
$$R:=(\grH)^{co\, \pi} = \{ h \in \grH \,: \, (\id \otimes \pi)\Delta(h)= h \otimes 1\}.$$
By a result of Radford \cite{R85} and Majid \cite{Mj}, $R$ is a graded braided Hopf algebra, that is, it is a Hopf algebra in the braided category $\YD$ of left Yetter-Drinfeld modules over $H_0=\k G$. Moreover, $\grH$ is the \textit{bosonization} (or \textit{Radford biproduct}) of $R$ and $H_0$ such that $\grH \cong R \# H_0$ with the following Hopf structure \cite[Theorem 10.6.5]{MO93}:

\begin{itemize}
 \item underlying vector space is the tensor product $R \otimes H_0$,
 \item multiplication $(r\#g)(r'\#g') = r(g \cdot r') \# (gg')$,
 \item comultiplication $\Delta(r\#g) = \sum r^{(1)} \# (r^{(2)})_{(-1)} g \otimes (r^{(2)})_{(0)} \# g$, 
 \item counit $\varepsilon(r \# g) = \varepsilon_R(r) \varepsilon_G(g)=\varepsilon_R(r)$,
 \item antipode $S(r\#g) = \sum (1 \# S_G(r_{(-1)}g))(S_R(r_{(0)}) \# 1)$,
\end{itemize} 
for all $r, r' \in R$ and $g, g' \in G$. Here, we use the Sweedler notation $\Delta_R(r) = \sum r^{(1)} \otimes r^{(2)}$ for the comultiplication of $R$; and the $H_0$-coaction on $R$, $\rho_R: R \rightarrow H_0 \otimes R$, is given by $\rho_R(r)=\sum r_{(-1)} \otimes r_{(0)}$. In this case, the Hopf algebra projection $\pi: R \# H_0 \rightarrow H_0$ and inclusion $\iota: H_0 \rightarrow R \# H_0$ are given by $\pi(r \# g) = \varepsilon(r)g$ and $\iota(g)= 1 \# g$, respectively. 

Our strategy, following the principle proposed in \cite{NA98}, is as follows: to study $H$, given $H_0=\k G$, we will first study $R$ as a Hopf algebra in $\YD$, then transfer its structures to $\grH$ via the bosonization $R \# H_0$, and finally we lift $\grH$ structures to $H$ via the filtration. We present the classification of pointed $p^3$-dimensional Hopf algebras $H$ in characteristic $p>0$ in Sections~\ref{sec:R} and \ref{sec:pointed p3} of this paper.


\subsection{Yetter-Drinfeld modules}
\label{subsec:YD}
We say that a vector space $V$ is a left Yetter-Drinfeld module over $\k G$ if $V$ is both a $\k G$-module and $\k G$-comodule satisfying 
$$\rho_V(g \cdot v) =\sum gv_{(-1)}g^{-1} \otimes g \cdot v_{(0)},$$
for all $v\in V$ and $g\in G$. We denote by $\rho_V: V\to \k G\otimes V$ the $G$-coaction on $V$ such that $\rho_V(v)=\sum v_{(-1)}\otimes v_{(0)}$, for any $v\in V$. Since $V$ is a $\k G$-comodule,  $V=\bigoplus_{g \in G} V_g$ such that for any homogenous element $v\in V_g$ we have $\rho_V(v)=g\otimes v$, for $g \in G$. We call $g$ the \textit{$G$-grading} of $v \in V_g$. In particular, for an abelian group $G$, any Yetter-Drinfeld module $V \in \YD$ is a $G$-graded $\k G$-module, i.e. $V = \bigoplus_{g \in G} V_g$ such that each $V_g$ is a $\k G$-submodule of $V$.

We denote by  $\YD$  the category of all left Yetter-Drinfeld modules over $\k G$, where the morphisms are both $G$-linear and $G$-colinear maps. In general, the category $\YD$ is a braided monoidal (tensor) category. That is, for any $V, W\in \YD$, $V \otimes W$ is again an object of $\YD$ via action $g \cdot (v \otimes w) = \Delta(g) (v \otimes w)= (g \cdot v) \otimes (g \cdot w)$, for all $g \in G, v \in V$, and $w \in W$. The braiding structure is given by twisting maps $c_{V,W}: V \otimes W \rightarrow W \otimes V$, which is an isomorphism in the category $\YD$:
$$c_{V,W}(v \otimes w) = \sum v_{(-1)} \cdot w \otimes v_{(0)}$$
with inverse 
$$c_{V,W}^{-1} (w \otimes v) = \sum v_{(0)} \otimes S^{-1}(v_{(-1)}) \cdot w,$$
for all $v \in V$ and $w \in W$. 

Let $A, B$ be two $\k$-algebras in the category $\YD$, the tensor product $A\otimes B$ is again a $\k$-algebra in $\YD$ with multiplication given by $m_{A \otimes B}=(m_A \otimes m_B)(\id_A \otimes c_{B,A} \otimes \id_B)$ satisfying:
$$1_{A\otimes B} = 1_A \otimes 1_B \qquad \text{ and } \qquad (a \otimes b)(a' \otimes b') = a(b_{(-1)} \cdot a') \otimes b_{(0)}b'$$
for all $a, a' \in A$ and $b, b' \in B$.  

\subsection{Braided Hopf algebra $R$}
Following the above notations, we say that $R$ is a braided Hopf algebra in the category $\YD$ of Yetter-Drinfeld modules over group algebra $H_0=\k G$, with the following structures (c.f.~\cite[\S 1]{NA02}): 
\begin{itemize}
 \item The action of $H_0$ on $R$ is given by the adjoint action composed with the inclusion $\iota: H_0 \rightarrow \grH$. The coaction of $H_0$ on $R$ is given by $ \rho_R = (\pi \otimes \id)\Delta: R\to H_0 \otimes R$. 
 \item $R$ is a Hopf algebra in $\YD$: it is a subalgebra of $\grH$ and a coalgebra with comultiplication $\Delta_R(r)=r_{(1)}\iota \pi S(r_{(2)}) \otimes r_{(3)}$ and antipode $S_R(r) = \iota \pi(r_{(1)})S(r_{(2)})$, for all $r \in R$. 
\end{itemize} 

Observe that $R$ inherits the usual $\mathbb{N}$-grading from $\grH$, so $R= \bigoplus_{n \geq 0} R_n$ is a graded braided Hopf algebra, where each $R_n:=R \cap (\grH)_n$ is a Yetter-Drinfeld submodule of $R$ over $H_0= \k G$. Let $V=\{ r \in R \,| \, \Delta_R(r) = r \otimes 1 + 1 \otimes r\}$ be the space of primitive elements of $R$. By the discussion in \cite[\S 2]{NA98}, we have $\grH$ is  coradically graded, and it follows that $V=R_1$, that is, all primitive elements of $R$ has degree 1 in $R$ \cite[Lemma 2.4]{NA98}. Moreover, $V$ is a Yetter-Drinfeld submodule of $R$ and its braiding $c:=c_{V,V}: V \otimes V \rightarrow V \otimes V$ is called the \textit{infinitesimal braiding} of $H$. The dimension of $V$ is called the \textit{rank} of $H$. The subalgebra $\BV$ of $R$ generated over $\k$ by $V$ is a braided Hopf subalgebra whose structure only depends on the infinitesimal braiding of $V$, called the \textit{Nichols algebra} of $V$, c.f.~\cite[\S 2]{NA02}. 

We remark here that the structure of $R$ in characteristic $p>0$ is different than that in characteristic zero: There is a long-standing conjecture by Andruskiewitsch and Schneider, e.g.~\cite[Conjecture 2.7]{NA02}, that in characteristic zero, if $H_0$ is an abelian group algebra then $R$ is always primitively generated, $R = \langle R_0 \oplus R_1 \rangle$. This is proved by Angiono and Iglesias in \cite[Theorem 2.6]{AI} for the case in which the infinitesimal braiding is of standard type. Thus, in such case, $R=\BV$, the Nichols algebra of the Yetter-Drinfeld submodule $V=R_1$. However, in characteristic $p>0$, $R$ may not be primitively generated in general and there may exist non-primitive generator(s) in $R$, e.g.~\cite[Example 2.5]{NA02}. We will consider an example of this case later in Section~\ref{subsec:non-prim}.


\subsection{Simple and indecomposable $\YD$-modules}
\label{subsec:simpleYD}

Let $\k$ be a base field of arbitrary characteristic and $G$ be any finite group. We look at the simple and indecomposable Yetter-Drinfeld modules over $\k G$. This work has been previously discussed by other authors such as \cite[Proposition 3.1.2]{AG99} and \cite[Corollary 3.2]{KMM}. We present here a ring-theoretical approach, independent of the characteristic of the field $\k$. 

Recall that the \textit{Drinfeld (quantum) double} $D(H)$ of any finite-dimensional Hopf algebra $H$ over $\k$ is again a Hopf algebra and is defined as follows. As a coalgebra, $D(H) = H^{*cop} \otimes H$, where $H^*=\Hom_\k(H,\k)$ is the vector space dual of $H$ which is again a Hopf algebra, and the comultiplication of $H^{*cop}$ is opposite to that of $H^*$. Denote $f \otimes x$ by $f \Join x \in D(H)$, for any $f \in H^{*cop} = H^*$ and $x \in H$, then the multiplication in $D(H)$ is given by
$$(f \Join x)(g \Join y) = \sum f(x_{(1)} \rightharpoonup g \leftharpoonup S^{-1}(x_{(3)})) \Join (x_{(2)}y),$$
for any $f, g \in H^{*cop} = H^*$ and $x,y \in H$, where $x \rightharpoonup g$ and $g \leftharpoonup x \in H^{*cop} = H^*$ are defined as:
$$\langle x \rightharpoonup g, y \rangle = \langle g, yx \rangle, \qquad \langle g \leftharpoonup x, y \rangle = \langle g, xy \rangle.$$

For any finite group $G$, we have $D(G):= D(\k G) = \k^G \# \k G$, where $\k^G := (\k G)^*= \text{Span}\{ \delta_g \,| \, g \in G\}$, with action of $\k G$ on $\k^G$ via $g \cdot \delta_h = \delta_{ghg^{-1}}$, for any $g \in G$ and $\delta_h \in \k^G$. In particular, $D(G)$ is generated by all choices of $g=(\varepsilon \Join g)$ and $\delta_h = (\delta_h \Join 1_G)$, for $g, h \in G$, satisfying $g \delta_h = (\varepsilon \Join g)(\delta_h \Join 1_G)= \delta_{ghg^{-1}} \Join g$.

Let $G= \sqcup_{i=1}^n \O_i$ be the conjugacy class decomposition of $G$. We see that $\k^{\O_i} = \text{Span}\{ \delta_h \, | \, h \in \O_i \}$ is a subring of $\k^G$. Clearly from the definition, we have 

\begin{lem}
\label{lem:DG}
$D(G) \cong \bigoplus_{i=1}^n (\k^{\O_i} \# \k G)$ as finite-dimensional algebras. 
\end{lem}

Let $R$ be any ring and $e \in R$ be an idempotent of $R$ such that $ReR=R$. 

\begin{lem}
\label{lem:Morita}
The two categories $R$-Mod and  $eRe$-Mod are Morita equivalent. 
\end{lem}

\begin{proof}
Consider the Schur functor $F: R$-Mod $\rightarrow eRe$-Mod and its natural inverse $G: eRe$-Mod $\rightarrow R$-Mod, defined via $F: M \mapsto eR \otimes_R M$ and $G: N \mapsto Re \otimes_{eRe} N$. It is easy to check from the construction that $GF = \id_{R-\text{Mod}}$ and $FG = \id_{eRe-\text{Mod}}$. 
\end{proof}

\begin{lem} In each finite-dimensional algebra $R=\k^{\O_i} \# \k G$, we denote idempotent $e=(\delta_g \# 1_G)$, for any choice of $g \in \O_i$. Then:
\begin{enumerate}[(a)]
\item $1_R=(\sum_{h\in \O_i}\delta_h)\# 1_G$.
 \item $ReR=R$.
 \item $eRe \cong \k C(g)$, where $C(g)$ is the centralizer of $g$ in $G$. 
\end{enumerate}
\end{lem}

\begin{proof} Fix any choice of $g \in \O_i$, it is clear that $e=(\delta_g \# 1_G)=(\delta_g \# 1_G)(\delta_g \# 1_G)=e^2$ is an idempotent. 

(a) Write $\id=\sum_{h\in \O_i}\delta_h$. It is clear that $\id\delta_h=\delta_h\id=\delta_h$ for any $h\in \O_i$. Thus $\id\#1_G$ is the identity. 

(b) Observe that for any $h, h' \in G$, $(\id \# h)(\delta_g \# 1_G)(\id \# h')=(\delta_{hgh^{-1}} \# hh')$. By varying all choices of $h,h' \in G$, this identity will give us all of $(\k^{\O_i} \# \k G)$.

(c) For any $\delta_h \in \k^{\O_i}$ and $h' \in G$,
\begin{align*} 
(\delta_g \# 1_G)(\delta_h \# h')(\delta_g \# 1_G) &= \delta_g \delta_h \delta_{h'g(h')^{-1}} \# h' \\ &= \begin{cases} 0, & \quad h \neq g \text{ or } h'g(h')^{-1} \neq g, \\
\delta_g \# h', & \quad h=g \text{ and } h' \in C(g).
\end{cases} \end{align*}
Since $g$ is fixed, by identifying $\delta_g \# h' \leftrightarrow h'$, we obtain the statement as claimed.
\end{proof}

The following results follow from Lemmas~\ref{lem:DG} and \ref{lem:Morita}:

\begin{cor}\label{C:SImodule}
\begin{enumerate}[(a)]
 \item The two categories $(\k^{\O_i} \# \k G)$-Mod and $\k C(g)$-Mod are Morita equivalent, for any $g \in \O_i$.
 \item $D(G)$-Mod is Morita equivalent to $\left(\bigoplus_{i=1}^n \k C(g_i)\right)$-Mod, for any choice of $g_i \in \O_i$.
\end{enumerate}
As a consequence, there are bijections between
$$\displaystyle \{ \text{simple } D(G)\text{-modules} \} \longleftrightarrow \sqcup_{1 \le i \le n} \left\{ \text{simple } \k C(g_i)\text{-modules} \right\},$$
and
$$\displaystyle \{ \text{indecomposable } D(G)\text{-modules} \} \longleftrightarrow \sqcup_{1 \le i \le n} \left\{ \text{indecomposable } \k C(g_i)\text{-modules} \right\}.$$
\end{cor}

It is well-known that the category of Yetter-Drinfeld modules $\YD$ is equivalent to the category of left $D(G)$-modules for any finite group $G$. The above bijections can then be written in details so that every simple (resp.~indecomposable) $\YD$-module is given in the form of $\k G \otimes_{C(g_i)} M$, for some $g_i \in \O_i$ and some simple (resp.~indecomposable) module $M$ over $C(g_i)$. Moreover, the $\k G$-module structure is given as the induced module ($\k G \otimes_{C(g_i)} - )$ and the $\k G$-comodule structure is given by $\rho(x \otimes m) = xg_ix^{-1} \otimes (x \otimes m)$, for any $(x \otimes m) \in \k G \otimes_{C(g_i)} M$.

\begin{remark}
\label{simple1dim}
Let $G$ be a finite $p$-group, and char. $\k = p>0$. Then every simple module in $\YD$ is of the form $\k G \otimes_{C(g_i)} \k$, where $\k$ is the trivial module over $C(g_i)$, for some choice $g_i \in \O_i$ and $\rho(x \otimes 1) = xg_ix^{-1} \otimes (x \otimes 1)$, for any $x \in G$. 
In particular, if $G=C_p$ is a cyclic group of prime order $p$ and char. $\k=p$, then every simple $\YD$-module is one-dimensional with trial $G$-action and some $G$-grading. 
\end{remark}


\subsection{Braiding types of $\YD$-modules and Nichols algebras.}
\label{subsec:Nichols}

A full description of Nichols algebras is given in e.g.~\cite[\S 2 and \S 3]{NA02}. Here, we only focus on the case with $H_0=\k G$ and $G$ is a finite group.  
Let $V$ be a Yetter-Drinfeld module in $\YD$ of finite dimension $d$ with fixed basis $\{x_1, x_2, \ldots, x_d\}$. Let $\chi: G \rightarrow \k$ be any character on $G$. We denote
$$V_g := \{x  \in V \, | \, \rho(x) = g \otimes x \}, \quad V^{\chi} := \{x \in V  \, | \, g \cdot x = \chi(g)x, g \in G\}.$$
The support of $V$ is defined to be
$$\text{supp}(V):=\{g\in  G\, |\, V_g\neq 0\}$$
and we denote by $G_V$ the subgroup of $G$ generated by $\text{supp}(V)$. Note that we can consider $V$ as a Yetter-Drinfeld module over $G_V$ by natural restriction. 

There are two types of braiding $c: V \otimes V \rightarrow V \otimes V$ in which we are interested in this paper: 

\begin{itemize}
 \item $V = \D(d)$ is of \textit{diagonal type}: 
 $$c(x_i \otimes x_j) = q_{ij}x_j \otimes x_i,$$
for all $1\le i, j \in d$, where $q_{ij} \in \k^\times$. We call $\q = (q_{ij})_{1 \le i,j \le d}$ the braiding matrix. Moreover, we say that a braided vector space $(V,c)$ of diagonal type is \textit{realizable} as some Yetter-Drinfeld module in $\YD$ if $V$ admits a \textit{principal $\mathcal{YD}$-realization}, that is, if we can find a family $(g_i,\chi_i)_{1 \le i \le d}$ of elements $g_i \in Z(G)$ and characters $\chi_i \in \widehat{G}$ such that $\chi_j(g_i) = q_{ij} \in \k^\times$, for all $1 \le i,j \le d$. In this case, we can consider $V$ as a Yetter-Drinfeld module in $\YD$ by defining $x_i \in V_{g_i}^{\chi_i} = V_{g_i} \cap V^{\chi_i}$, for all $1 \le i \le d$, \cite[\S 2]{AI1}. 

 \item One important example of the non-diagonal types is the Jordan type. We say $V = \J(t,d)$ is of \textit{Jordan type}, which only happens when $d \geq 2$: 
 $$c(x_i \otimes x_1) = tx_1 \otimes x_i \qquad \text{and} \qquad c(x_i \otimes x_j)= (tx_j + x_{j-1}) \otimes x_i,$$ 
for all $1\le i\le d, 2\le j\le d$, and some $t \in \k^{\times}$. The braiding matrix is of the form: 
$$\begin{pmatrix}
   t & 0 & 0 & \cdots & 0 & 0 & 0 \\       
   1 & t & 0 & \cdots & 0 & 0 & 0 \\
   0 & 1 & t & \cdots & 0 & 0 & 0 \\   
   \vdots & & & \ddots & & & \vdots  \\ 
   0 & 0 & 0 & \cdots & t & 0 & 0 \\   
   0 & 0 & 0 & \cdots & 1 & t & 0 \\
   0 & 0 & 0 & \cdots & 0 & 1 & t
 \end{pmatrix}_{d\times d}$$ 
We say that the braided vector space $(V,c)$ of Jordan type is \textit{realizable} as some Yetter-Drinfeld module in $\YD$ if $V$ admits a \textit{non-principal $\mathcal{YD}$-realization}, that is, if we can find a pair $(\varphi,g)\in \text{rep}_d(G)\times Z(G)$ such that each basis $x_i$ has the same $G$-grading via coaction $\rho(x_i) = g \otimes x_i$, and $\varphi: G\to \text{GL}(V)$ is a representation of $G$ on $V$ satisfying $\varphi(g) \cdot x_1 = tx_1$ and $\varphi(g)  \cdot x_i = tx_i + x_{i-1}$ for all $2\le i\le d$. 
\end{itemize}
Note that the Nichols algebra $\BV$ is uniquely determined by the braiding $c$, so $\BV$ only depends on the (co)actions of $G$ on $V$ as described in above cases. 

\begin{prop}
\label{prop:EquivConDiagonalJordan}
Let $G$ be a finite group, and $(V,c)$ be a Yetter-Drinfeld module over $G$ of dimension $d < \infty$. Then
\begin{enumerate}
\item $(V,c)$ is of diagonal type if $V$ is a direct sum of one-dimensional Yetter-Drinfeld modules over $G_V$. Moreover, $V$ has a principal $\mathcal{YD}$-realization if and only if $V$ is a direct sum of one-dimensional Yetter-Drinfeld modules over $G$.
\item $(V,c)$ is of Jordan type if there is some central group element $g\in G$ and $t\in \k^\times$ such that all elements in $V$ have the same $G$-grading given by $g$ and the minimal polynomial for the $g$-action on $V$ is $(g-t)^d$, with $d= \dim V\ge 2$. 
\end{enumerate}
\end{prop}

\begin{proof}
(1) View $V$ as a Yetter-Drinfeld module over $G_V$. Suppose $V=\oplus_{1\le i\le d} V_i$ is a direct sum of one-dimensional Yetter-Drinfeld modules over $G_V$. Let $x_j$ be a basis for $V_j$, for all $1 \le j \le d$. Then one sees that, for any $V_j$, the $G_V$-action is given by some character $\chi_j: G_V\to \k$ such that $h\cdot x_j=\chi_j(h)x_j$, for any $h\in G_V$; and the $G_V$-coaction is given by some central element $g_j\in G_V$ such that $\rho(x_j)=g_j\otimes x_j$. As a consequence, $V$ is of diagonal type with braiding matrix given by $\q=(q_{ij})_{1\le i,j\le d}$ with $q_{ij}=\chi_j(g_i)$. The argument for principal $\mathcal{YD}$-realization is similar.

(2) Suppose there are $g\in Z(G)$ and $t\in \k^\times$ such that all elements in $V$ are of the same $G$-grading $g$ and the minimal polynomial for the $g$-action on $V$ is $(g-t)^d$ with $d=\dim V\ge 2$. It follows from the Jordan block decomposition that there exists a basis $x_1,\dots,x_d$ such that $g\cdot x_1=tx_1$ and $g\cdot x_j=tx_j+x_{j-1}$, for all $2\le j\le d$. Then it is clear that the braiding of $V$ associated to this basis is of Jordan type. 
\end{proof}


\subsection{Hochschild cohomology of Nichols algebras}
\label{subsec:Hochschild}

Let $G$ be a finite group, and $V$ be a Yetter-Drinfeld module over $G$. Denote by $\BV$ the Nichols algebra of $V$, which is a connected graded braided Hopf algebra in $\YD$. We will use the cobar construction on $\BV$ (cf. \cite{DZZ}), denoted by $\Omega \BV$,  to compute its Hochschild cohomology.
\begin{itemize}
\item  As graded algebras, $\Omega \BV = T(\BV^+)$ is the tensor algebra of $\BV^+:=\{ \text{all grading } \geq 1 \text{ parts of } \BV \}$.
\item The differentials $\partial^n: (\BV^+)^{\otimes n} \rightarrow (\BV^+)^{\otimes (n+1)}$ are given by
\[
\partial^n=\sum_{i=0}^{n-1}(-1)^{i+1}\, 1^{\otimes i} \otimes \overline{\Delta}\otimes 1^{\otimes (n-i-1)},
\]
where $\overline{\Delta}(r)=\Delta(r)-r\otimes 1-1\otimes r$, for any $r\in\BV^+$.
\end{itemize}
The differentials in low degrees are explicitly given as follows:
\begin{align*}
\partial^1(x) &= \Delta(x) - x \otimes 1 - 1 \otimes x, \\
\partial^2(x \otimes y) &= 1 \otimes x \otimes y - \Delta(x) \otimes y + x \otimes \Delta(y) - x \otimes y \otimes 1
\end{align*}
for $x,y\in \BV^+$. The $i$-th Hochschild cohomology of $\BV$ is defined to be 
\[
\HL^i(\BV,\k):=\HL^i(\Omega \BV,\partial)=\Ker\, \partial^i/\Img\, \partial^{i-1}.
\]

\begin{lem}\label{Cobar}
Let $V^*$ be the dual of $V$ in $\YD$. Suppose $\BV$ is finite-dimensional. Then $\HL^i(\BV,\k)=\HH^i(\B(V^*),\k)$, where $\HH^i(\B(V^*),\k)$ is the $i$-th Hochschild cohomology of $\B(V^*)$ with coefficient in the trivial module $\k$.
\end{lem}
\begin{proof}
Note that $\!^*V=V^*$ since $S^2=\id$ in $D(G)$. According to \cite[3.2.30]{AG99}, one sees that $\BV^{*bop}\cong \B(V^*)$ as connected graded braided Hopf algebras in $\YD$. Hence, by \cite[Proposition 1.4]{SteVan}, we have 
\[
\HL^i(\BV,\k)=\HH^i(\BV^*,\k)=\HH^i(\BV^{*bop},\k)=\HH^i(\B(V^*),\k).
\] 
\vspace{-0.5em}
\end{proof}

It is important to point out that there is a bigrading structure on the Hochschild cohomology of $\BV$, that is, the homological grading and the Adams grading induced from the grading of $\BV$. From now on, we write 
\[
\HL^i(\BV,\k)=\bigoplus_{j\ge 0} \HL^{i,j}(\BV,\k),
\]
where $\HL^{i,j}(\BV,\k)$ is the direct summand of $\HL^i(\BV,\k)$ in $\YD$ consisting of homogeneous elements whose Adams grading is $j$. Moreover, we have

\begin{lem}
The Hochschild cohomology ring $\bigoplus_{i,j\ge 0} \HL^{i,j}(\BV,\k)$ is a bigraded braided algebra in $\YD$. 
\end{lem}
\begin{proof}
The braided algebra structure can be computed from the cobar construction $\Omega \BV=T(\BV^+)$. One sees that $T(\BV^+)$ is a bigraded braided algebra in $\YD$ with the tensor degree (inducing the homological grading) and the inner degree of $\BV$ (inducing the Adams grading). One can check that all the differentials are module maps in $\YD$ since they are constructed from the comultiplication of $\BV$. Moreover, the differentials also preserve the bigrading of $\Omega \BV$. Hence, the result follows after taking cohomology.  
\end{proof}

\begin{thm}
\label{ExtNichols}
Let $R=\bigoplus_{i\ge 0} R_i$ be a connected, coradically graded, and braided Hopf algebra in $\YD$. Suppose $R$ is not primitively generated. Then the differential $\partial^1$ induces an injective Yetter-Drinfeld module map 
\[
\xymatrix{
\partial^1: R_d/\B(R_1)_d \ \ar@{^{(}->}[r] &  \HL^{2,d}(\B(R_1),\k),
}
\] 
where $d\ge 2$ is the smallest integer such that $\B(R_1)_d\subsetneq R_d$. 
\end{thm}
\begin{proof}
By the definition of Nichols algebra, $\B(R_1)$ is a graded Hopf subalgebra of $R$. Moreover, since $R$ is coradically graded, $\B(R_1)$ contains all primitive elements $R_1$ of $R$. Note that the integer $d\ge 2$ is well-defined for $R$ is not primitively generated. 

First, we show that $\partial^1$ is well-defined on $R_d/\B(R_1)_d$. Choose any $z\in R_d$. Without loss of generality, we can assume $z\in R_{\ge 1}$ and write 
$$\Delta(z) = z \otimes 1 + 1 \otimes z + \omega,$$
where $\omega \in \bigoplus_{j=1}^{d-1} R_{j} \otimes R_{d-j}$ since $R$ is a graded coalgebra. By the choice of $d$, $R_{j} \otimes R_{d-j} \subseteq \B(R_1) \otimes \B(R_1)$ for all $1\le j\le d-1$. Therefore, $\omega \in \B(R_1) \otimes \B(R_1)$. Moreover, the coalgebra structure of $R$ is coassociative, so $(\Delta \otimes \id) \Delta = (\id \otimes \Delta) \Delta$ on $z$. Thus, we have:
\begin{align*}
(\Delta \otimes \id)(z \otimes 1 + 1 \otimes z + \omega) &= \Delta(z) \otimes 1 + 1 \otimes 1 \otimes z + (\Delta \otimes \id)(\omega) \\
&= (z \otimes 1 + 1 \otimes z + \omega) \otimes 1 +  1 \otimes 1 \otimes z + (\Delta \otimes \id)(\omega); 
\end{align*}
\begin{align*}
(\id \otimes \Delta)(z \otimes 1 + 1 \otimes z + \omega) &= z \otimes 1 \otimes 1 + 1 \otimes \Delta(z) + (\id \otimes \Delta)(\omega) \\
&= z \otimes 1 \otimes 1 + 1 \otimes (z \otimes 1 + 1 \otimes z + \omega) +  (\id \otimes \Delta)(\omega),
\end{align*}
which implies $1 \otimes \omega - (\Delta \otimes \id)(\omega) + (\id \otimes \Delta)(\omega) - \omega \otimes 1 = 0$, equivalent to saying $\partial^1(z)=\omega$ is a 2-cocycle in $\Omega \BV$. By the construction of $\partial^1$, we can conclude that $\partial^1:z\mapsto \omega$ induces a map from $R_d$ to $\HL^{2,d}(\B(R_1),\k)$, which is also a Yetter-Drinfeld module map and preserves the Adams grading of $\BV$. Furthermore, the map $\partial^1$ factors through the Yetter-Drinfeld submodule $B(R_1)_d$ since $\partial^2\partial^1=0$ in $\Omega \B(R_1)$. 

Next we show that $\partial^1$ is injective after the factorization. Suppose $\partial^1(z)=\omega=\partial^1(x)$, for some $x \in \B(R_1)^+$. Then we have $\Delta(z) - z \otimes 1 - 1 \otimes z = \omega = \Delta(x) - x \otimes 1 - 1 \otimes x$. Hence $\Delta(z-x) = (z-x) \otimes 1 + 1 \otimes (z-x)$, so $(z-x)$ is a primitive element in $R$ and it belongs to $R_1\subset \B(R_1)^+$. So $z=x+(z-x) \in \B(R_1)^+$, which yields that $z\in \B(R_1)\cap R_d=\B(R_1)_d$. This proves the theorem.
\end{proof}

Therefore, the second Hochschild cohomology group $\HL^2(\B(R_1),\k)$ controls the non-primitive generators, if any, of a connected coradically graded braided Hopf algebra $R$ in $\YD$. We obtain the following result.

\begin{cor}
Let $R$ be as stated as in \Cref{ExtNichols}. If $\HL^2(\B(R_1),\k)=0$, then $\B(R_1)=R$. As a consequence, if $R$ is not primitively generated, then $\HL^2(\B(R_1),\k)$ does not vanish.   
\end{cor}

Since $\HL^n(\Omega \B(V))$ is a Yetter-Drinfeld module over $\k G$, we denote by $\HL^n(\Omega \B(V))^{co G}$ the coinvariants of $\k G$ in $\HL^n(\Omega \B(V))$, that is the trivial $G$-grading part. Note that $\bigoplus_{n\ge 0}\HL^n(\Omega \B(V))^{co G}$ is a subalgebra of $\bigoplus_{n\ge 0}\HL^n(\Omega \B(V))$.

\begin{prop}\label{Hsmash}
Suppose $\B(V)$ is finite-dimensional. Then $\HL^n(\Omega(\B(V)\#\k G))=\HL^n(\Omega \B(V))^{co G}$, for all $n\ge 0$. As a consequence, 
$$\bigoplus_{n\ge 0} \HL^n(\Omega(\B(V)\#\k G))\cong \bigoplus_{n\ge 0}\HL^n(\Omega \B(V))^{co G}$$ as algebras. 
\end{prop}
\begin{proof}
By \cite[\S 2.2]{AG99}, we know $B(V)^*$ is a braided Hopf algebra over $(\k G)^*$, and $(\B(V)\#\k G)^*\cong \B(V)^*\# (\k G)^*$ as Hopf algebras. There is a well-known spectral sequence of algebras, for example see reference \cite{Sten},
\[
\HH^p((\k G)^*, \HH^q(\B(V)^*,\k))\Longrightarrow \HH^{p+q}(\B(V)^*\# (\k G)^*, \k).
\]
Since $(\k G)^*$ is semisimple, we have 
\[
\HH^q(\B(V)^*,\k)^{(\k G)^*}\cong \HH^{q}(\B(V)^*\# (\k G)^*, \k)
\]
as algebras. Therefore, by the same argument in Lemma \ref{Cobar}, we get
\begin{align*}
\HL^n(\Omega (\B(V)\# \k G))&\, \cong \HH^n((\B(V)\# \k G)^*, \k)\\
&\, \cong \HH^n(\B(V)^*\# (\k G)^*, \k)\\
&\, \cong \HH^n(\B(V)^*,\k)^{(\k G)^*}\\
&\, \cong \HL^n(\Omega(\B(V))^{co G}.
\vspace{-0.5em}
\end{align*}
\end{proof}


\section{Nichols algebras of infinitesimal braidings of rank two}
\label{sec:BVrank2}


\subsection{Indecomposable $2$-dimensional $\YD$ modules}

In this section, we work over an algebraically closed field $\k$ of any characteristic. Let $G$ be a finite group.

There are two different ways to construct a two-dimensional Yetter-Drinfeld module $V$ over $G$. One is to choose a pair $(\varphi,g)\in \text{rep}_2(G)\times Z(G)$. Then $V$ is constructed as a two-dimensional Yetter-Drinfeld module over $G$, where the $G$-action on $V$ is given by $\varphi: G\to GL_2(V)$ and $V$ has constant $G$-grading given by $g$. On the other hand, let $g\in G$ such that the centralizer $C(g)$ is of index two in $G$. Then for any character $\chi: C(g)\to \k$, take $V=\k G\otimes_{C(g)}\chi$, where the $G$-action is the induced action on $\chi$ and the $G$-coaction on $x\otimes \chi$, for any $x\in G$, is given by $xgx^{-1}$. In both constructions, we denote by $V=(\varphi,g)\in\text{rep}_2(G)\times Z(G)$ or $V=(\chi,g)\in \widehat{C(g)} \times G$. 

\begin{lem}
\label{L:Indecomposable}
Suppose $V \in \YD$ is indecomposable of dimension $2$, then $V$ is isomorphic to one of the following:
\begin{enumerate}
\item  $V=(\varphi,g)\in\text{rep}_2(G)\times Z(G)$, where $\varphi$ is a two-dimensional indecomposable representation of $\k G$.
\item  $V=(\chi,g)\in \widehat{C(g)} \times G$, where $C(g)$ has index two in $G$.
\end{enumerate}
\end{lem}
\begin{proof}
By Corollary \ref{C:SImodule} and the discussion below it, any indecomposable Yetter-Drinfeld module over $G$ is given in terms of $V=\k G\otimes_{C(g)}M$, where $M$ is a indecomposable module over $C(g)$. Suppose $\dim V=2$. Then either $\dim M=2$, in this case $C(g)=G$ so $g$ is central; or $\dim M=1$ and $|G:C(g)|=2$. The result follows.  
\end{proof}

In the following, we denote by $\mathcal M_G^2$ the isomorphism classes of all two-dimensional indecomposable modules over $\k G$. Note that when $\k=\overline{\k}$ and $\text{char}\, \k\nmid |G|$, $\mathcal M_G^2$ contains all the $2\times 2$ matrix blocks of the semisimple algebra $\k G$. In the conjugacy class decomposition $G= \sqcup_{1\le i\le n} \O_i$, for any $g_i\in \O_i$, $\widehat{C(g_i)}$ denotes all the characters on the centralizer $C(g_i)$. Note that $\widehat{C(g_i)}$ is bijective to the set of grouplike elements in the dual Hopf algebra $\k^{C(g_i)}$. The next result is straightforward from previous argument. 

\begin{prop}
The set of isomorphism classes of two-dimensional indecomposable Yetter-Drinfeld modules over $G$ is bijective to the following set:
$$ \left( \mathcal M_G^2\times Z(G) \right) \bigsqcup \left( \sqcup_{|\O_i|=2} \widehat{C(g_i)} \right).$$
\end{prop}


\subsection{Infinitesimal braidings of $2$-dimensional $\YD$ modules}

We discuss infinitesimal braidings of a two-dimensional Yetter-Drinfeld module $V$ over $G$. Let $\mathcal O_i$ be an orbit in the conjugacy class decomposition $G= \sqcup_{1\le i\le n} \O_i$ satisfying  $|\mathcal O_i|=2$. For any $g\in \mathcal O_i$ and any character $\chi:  C(g)\to \k$, we can construct a two-dimensional indecomposable Yetter-Drinfeld module $\k G \otimes_{C(g)} \chi$ over $G$. We denote it by $V=(\chi,g)$.

\begin{lem}
\label{L:diagonal}
Suppose $V=(\chi,g)\in \widehat{C(g)} \times G$, where $C(g)$ has index two in $G$. Then $V$ is of diagonal type with braiding matrix $(\chi(g))_{2\times 2}$.
\end{lem}
\begin{proof}
Since $|G:C(g)|=2$, there exists an element $s\in G\setminus C(g)$ such that $G=C(g)\sqcup C(g)s=C(g)\sqcup sC(g)$. Therefore $V=\k G\otimes_{C(g)}\chi$ has a basis $x_1=1_G\otimes \chi$ and $x_2=s\otimes \chi$. Regarding the $G$-coaction on $V$, one sees that
\[
\rho(x_1)=1_Gg1_G^{-1}\otimes (1_G\otimes \chi)=g\otimes x_1,\quad
\rho(x_2)=sgs^{-1}\otimes (s \otimes \chi)=sgs^{-1}\otimes x_2.
\]
The $G$-action on $x_1$ is given by
\begin{align*}
h\cdot x_1&\, =h\cdot (1_G\otimes \chi)=1_G\otimes h\cdot \chi=\chi(h)(1_G\otimes \chi)=\chi(h)x_1,\ \text{for any}\ h\in C(g);\ \text{and}\\
s\cdot x_1&\, =s\cdot(1_G\otimes \chi)=s\otimes \chi=x_2.
\end{align*}
And $G$ acts on $x_2$ via 
\[
h\cdot x_2=h\cdot (s\otimes \chi)=s\otimes h'\cdot \chi=\chi(h')(s\otimes \chi)=\chi(h')x_2=\chi(h)x_2,\ \text{for any}\ h\in C(g),
\]
where $hs=sh'$ for some $h'\in G$ and $\chi(h)=\chi(h')$; and 
\[
s\cdot x_2=s\cdot (s\otimes \chi)=s^2\cdot (1_G\otimes \chi)=1_G\otimes s^2\cdot \chi=\chi(s^2)(1_G\otimes \chi)=\chi(s^2)x_1,
\]
since $s^2\in C(g)$. In particular, we have 
\[
sgs^{-1}\cdot x_1=\chi(hgh^{-1})x_1=\chi(g)x_1,
\]
since $sgs^{-1}\in C(g)$; and
\[
sgs^{-1}\cdot x_2=sgs^{-1}\cdot (s \otimes \chi)= sg \cdot(1_G \otimes \chi)=sg\cdot x_1=\chi(g)s\cdot x_1=\chi(g)x_2.
\]
Hence, we see that $V$ is of diagonal type and the braiding matrix is $\q=(\chi(g))_{2\times 2}$.
\end{proof}

\begin{prop}
\label{prop:InfiniBraiding}
Let $V$ be a two-dimensional Yetter-Drinfeld module over $G$. Then the infinitesimal braiding of $V$ is either of diagonal type or of Jordan type. Moreover, we have 
\begin{enumerate}
\item If $V$ is decomposable, then $V$ is of diagonal type.
\item If $V=(\varphi,g)\in\text{rep}_2(G)\times Z(G)$, then $V$ is of diagonal type provided $\varphi(g)$ is diagonalizable on $V$, otherwise $V$ is of Jordan type. 
\item If $V=(\chi,g)\in \widehat{C(g)} \times G$, then $V$ is of diagonal type.
\end{enumerate} 
\end{prop}
\begin{proof}
(1) Suppose $V$ is decomposable. Then $V=V_1\oplus V_2$, where each $V_i$ is a one-dimensional $\YD$ module. So $V$ is of diagonal type by Proposition~\ref{prop:EquivConDiagonalJordan}.

(2) Suppose $V=(\varphi,g)\in\text{rep}_2(G)\times Z(G)$, where $\varphi: G\to \text{GL}(V)$ is a two-dimensional representation of $G$. Clearly, $V$ can be viewed as a Yetter-Drinfeld module over the subgroup $\langle g\rangle$ of $G$. We use the fact that $\dim V=2$. If $\varphi(g)$ is diagonalizable on $V$, then $V=V_1\oplus V_2$, where each $V_i$ is a one-dimensional Yetter-Drinfeld module over $\langle g\rangle$. Hence, by Proposition~\ref{prop:EquivConDiagonalJordan}, $V$ is of diagonal type.  Otherwise, the minimal polynomial of $\varphi(g)$ must be $(g-t)^2=0$ on $V$, for some $t\in \k^\times$. Then it again follows from Proposition~\ref{prop:EquivConDiagonalJordan} that $V$ is of Jordan type. 

(3) comes from Lemma \ref{L:diagonal}. 

Now according to Lemma \ref{L:Indecomposable}, we conclude that $V$ is either of diagonal type or of Jordan type.  
\end{proof}


\subsection{Nichols algebras of infinitesimal braidings in positive characteristic}
In the remaining of our paper, let $\k$ be an algebraically closed field of characteristic $p>0$. We use the convention $(x)(\ad \ y):=[x,y]$ (resp.~$(\ad\ y)(x):=[y,x]$) for the right (resp. left) adjoint action. The following proposition is used frequently in positive characteristic. 

\begin{prop}\cite[pp.~186-187]{Jac} 
\label{palgebra}
Let $A$ be any associative $\k$-algebra over a field of characteristic $p>0$. For any $x,y\in A$, we have
\begin{equation*}
\left(x+y\right)^{p}=x^{p}+y^{p}+\sum_{i=1}^{p-1} s_i\left(x,y\right)
\end{equation*}
where $is_i(x,y)$ is the coefficient of $\lambda^{i-1}$ in $x(\ad\, (\lambda x+y))^{p-1}$, and
\begin{equation*}
[x,y^p]=(x)(\ad\ y)^p.
\end{equation*}
\end{prop}

Let $G$ be a finite $p$-group. We are interested in Nichols algebras $\BV$ for any two-dimensional Yetter-Drinfeld module $V$ over $\k G$. Our next result is a restatement of \Cref{prop:InfiniBraiding} in characteristic $p>0$. 
\begin{lem}
\label{L:PositiveYDRank2}
Let $G$ be a finite $p$-group and $V$ be a two-dimensional Yetter-Drinfeld module over $G$. By choosing a suitable basis $\{x_1,x_2\}$ for $V$, we have:
\begin{enumerate}
\item If $V$ is decomposable, then $V=(g_1,g_2,\varepsilon)$, where $G$ coacts on $x_i$ by some $g_i\in G$ and the $G$-action is trivial.
\item If $V=(\varphi,g)\in \text{rep}_2(G)\times Z(G)$, then either $g$ acts trivially or $g\cdot x_1=x_1$, $g\cdot x_2=x_2+x_1$. 
\item If $p=2$ and $V=(\varepsilon,g)\in \widehat{C(g)} \times G$ with $|G:C(g)|=2$, then, for some $s\in G\setminus C(g)$, $G$ coacts on $x_1$ by $g$ and coacts on $x_2$ by $sgs^{-1}$; and the $C(g)$-action is trivial and $s\cdot x_1=x_2$, $s\cdot x_2=x_1$. 
\end{enumerate}
\end{lem}

\begin{proof}
Note that for $G$ (or any subgroup of G) we only have the trivial character $\varepsilon: G\to \k$. Moreover, since $G$ is a finite $p$-group and char. $\k=p>0$, we have $g^{|G|}-1=(g-1)^{|G|}=0$. Hence in the Jordan type, the minimal polynomial for $g$ is given by $(g-1)^2=1$. Also note that case (3) only happens in $p=2$ since $2\mid |G|$. Then we can apply Lemma \ref{L:diagonal} and Proposition~\ref{prop:InfiniBraiding}. 
\end{proof}

Next we focus on two-dimensional Yetter-Drinfeld modules $V$ of \textit{Jordan type}. By Lemma \ref{L:PositiveYDRank2}, there is a basis $x_1,x_2$ for $V$ and a central group element $g\in Z(G)$ such that 
\[
g\cdot x_1=x_1,\quad g\cdot x_2=x_2+x_1;
\]
and both $x_i$ have the same $G$-grading given by $g$. We denote by $\k\langle V\rangle$ the free braided Hopf algebra in $\YD$ generated by $V$, where $x_1,x_2$ are primitive elements. It is easy for one to check that $x_1^p$ is primitive in $\k\langle V\rangle$, i.e., $\Delta(x_1^p) = 1 \otimes x_1^p + x_1^p \otimes 1$.

\begin{lem}
\label{L:primitiveJordan}
Suppose $p>2$. In the braided Hopf algebra $\k\langle V\rangle$, we have $x_1x_2-x_2x_1-\frac{1}{2}x_1^2$ is primitive. Moreover when we pass to the quotient braided Hopf algebra $\k\langle V\rangle/(x_1^p,x_1x_2-x_2x_1-\frac{1}{2}x_1^2)$, $x_2^p$ is primitive. 
\end{lem}
\begin{proof}
Clearly, we have
\begin{align*}
& \Delta(x_1x_2-x_2x_1)\\
=\,&(x_1\otimes 1+1\otimes x_1)(x_2\otimes 1+1\otimes x_2)-(x_2\otimes 1+1\otimes x_2)(x_1\otimes 1+1\otimes x_1)\\
=\,&(x_1x_2-x_2x_2)\otimes 1+g\cdot x_2\otimes x_1-g\cdot x_1\otimes x_2+1\otimes (x_1x_2-x_2x_1)\\
=\,&(x_1x_2-x_2x_2)\otimes 1+x_1\otimes x_1+1\otimes (x_1x_2-x_2x_1).
\end{align*}
Subtracting $\Delta(\frac{1}{2}x_1^2)=\frac{1}{2}x_1^2\otimes 1+x_1\otimes x_1+1\otimes \frac{1}{2}x_1^2$ from the above equation, we deduce that $x_1x_2-x_2x_1-\frac{1}{2}x_1^2$ is primitive.

Now consider the quotient braided Hopf algebra $A:=\k\langle V\rangle/(x_1^p,x_1x_2-x_2x_1-\frac{1}{2}x_1^2)$. By induction on $n \ge 1$, we have
\[
x_1^nx_2-x_2x_1^n=\frac{n}{2}x_1^{n+1}.
\]
Next we prove the following identity inductively 
\begin{align}
\label{E:adjoint}
(x_2\otimes 1)(\ad\, (\lambda\, x_2\otimes 1+1\otimes x_2))^n=-\frac{(n+1)!}{2^n}\lambda^{n-1}\, x_1^n\otimes x_2,
\end{align}
for all $n\ge 1$ in the braided algebra $A\otimes A$.  First, one checks that it holds for $n=1$:
\begin{align*}
[x_2\otimes 1, \lambda x_2\otimes 1+1\otimes x_2]&\, =\lambda x_2^2\otimes 1+x_2\otimes x_2-\lambda x_2^2\otimes 1-g\cdot x_2\otimes x_2\\
&\, =-x_1\otimes x_2.
\end{align*}
Suppose it is true for $n=m$. Then for $n=m+1$, we have 
\begin{align*}
&\, (x_2\otimes 1)(\ad\, (\lambda\, x_2\otimes 1+1\otimes x_2))^{m+1}\\
&\, =[-\frac{(m+1)!}{2^m}\lambda^{m-1}\, x_1^m\otimes x_2,\lambda x_2\otimes 1+1\otimes x_2]\\
&\, =-\frac{(m+1)!}{2^m}\lambda^{m-1}\, [x_1^m\otimes x_2,\lambda x_2\otimes 1+1\otimes x_2]\\
&\, =-\frac{(m+1)!}{2^m}\lambda^{m-1}\, (\lambda x_1^m(g\cdot x_2)\otimes x_2+x_1^m\otimes x_2^2-\lambda x_2x_1^m\otimes x_2-g\cdot(x_1^m)\otimes x_2^2)\\
&\, =-\frac{(m+1)!}{2^m}\lambda^{m-1}\, (\lambda x_1^{m+1}\otimes x_2+\lambda (x_1^mx_2-x_2x_1^m)\otimes x_2)\\
&\, =-\frac{(m+1)!}{2^m}\lambda^{m-1}\, \lambda (1+\frac{m}{2})x_1^{m+1}\otimes x_2\\
&\, =-\frac{(m+2)!}{2^{m+1}}\lambda^{m}\, x_1^{m+1}\otimes x_2.
\end{align*}
This proves \eqref{E:adjoint}. Now we apply Proposition \ref{palgebra}. Then
\[
\Delta(x_2^p)=(x_2\otimes 1+1\otimes x_2)^p=x_2^p\otimes 1+1\otimes x_2^p+\sum_{i=1}^{p-1}s_i(x_2\otimes 1,1\otimes x_2),
\]
where $is_i(x_2\otimes 1,1\otimes x_2)$ is the coefficient of $\lambda^{i-1}$ in \eqref{E:adjoint} by letting $n=p-1$. One sees that all terms of $\lambda^{i-1}$ vanish as $n=p-1$. Hence $x_2^p$ is primitive. 
\end{proof}

\begin{prop}
\label{NicholsAlgebraRank2}
Let $V$ be a two-dimensional Yetter-Drinfeld module over $\k G$ with basis $\{x_1,x_2\}$.
\begin{enumerate}
\item If $V$ is of diagonal type, then $\BV=\k[x_1,x_2]/(x_1^p,x_2^p)$.
\item If $V$ is of Jordan type, then 
\[
\BV=
\begin{cases}
\k\langle x_1,x_2\rangle/(x_1^2,x_2^4,x_2^2x_1+x_1x_2^2+x_1x_2x_1,x_1x_2x_1x_2+x_2x_1x_2x_1), &  p=2\\
\k\langle x_1,x_2\rangle/(x_1^p,x_2^p,x_1x_2-x_2x_1-\frac{1}{2}x_1^2), & p>2.
\end{cases}
\]
\end{enumerate}
In both cases, $V=\text{Span}\{x_1,x_2\}$ is the primitive space of $\BV$.
\end{prop}

\begin{proof}
(1) If $V$ is of diagonal type, then it must have trivial braiding by Lemma \ref{L:PositiveYDRank2}. Following similar computation argument as in Lemma \ref{L:primitiveJordan}, we can check that $x_1^p,x_2^p$ and $[x_1,x_2]$ are all primitive. Hence the Nichols algebra $\BV$ is given by $\k[x_1,x_2]/(x_1^p,x_2^p)$. 

(2) The Jordan type can be derived from \cite[Theorem 3.1 and Theorem 3.5]{CLW}. Here we give a short proof for the case $p>2$. By Lemma \ref{L:primitiveJordan}, one sees that $\BV$ is a braided Hopf algebra quotient of $B:=\k\langle x_1,x_2\rangle/(x_1^p,x_2^p,x_1x_2-x_2x_1-\frac{1}{2}x_1^2)$. Hence, it suffices to show that any braided Hopf ideal $0\neq I\subset B$ must have $I\cap V\neq 0$. It is easy to check that $\dim B=p^2$. Since Nichols-Zoeller's freeness theorem \cite{NZ} still holds for $B$, $(\dim B/I)$ must divide $p^2$. Suppose $I\cap V=0$. Then there would be embeddings for the two Hopf subalgebras $\k[x_1]$ and $\k[x_2]$ into $B/I$. This implies that $\dim B/I = p^2$ and $I=0$. Hence we obtain a contradiction. By the definition of $\BV$, we have $\BV=B$.    
\end{proof}

We make a remark on $\dim \BV$ for the Jordan type in case (2) of Proposition \ref{NicholsAlgebraRank2}. When $p=2$, $\BV$ is a sixteen-dimensional algebra with basis
\begin{gather*}
1,\, x_1,\, x_2,\, x_1x_2,\, x_2x_1,\, x_2^2,\, x_1x_2x_1,\, x_1x_2^2,\\
x_2x_1x_2,\, x_2^3,\, x_1x_2x_1x_2,\, x_1x_2^3,\, x_2x_1x_2^2,\, x_1x_2x_1x_2^2,\, x_2x_1x_2^3,\, x_1x_2x_1x_2^3.
\end{gather*}
When $p>2$, $\BV$ is of dimension $p^2$ with basis $\{x_1^ix_2^j\, |\, 0\le i,j\le p-1\}$.

\begin{cor}
Let $\k$ be an algebraically closed field of characteristic $p>0$, and $V$ be a Yetter-Drinfeld module over a finite $p$-group $G$. Then if $\dim V\le 2$, we have $\dim \BV<\infty$.  
\end{cor}
\begin{proof}
The case of $\dim V=1$ is easy to check and the case of $\dim V=2$ follows from Proposition \ref{NicholsAlgebraRank2}. 
\end{proof}


\section{Pointed $p^3$-dimensional Hopf algebras: structures of $R$ and of $\grH$}
\label{sec:R}

For the rest of the paper, we let our base field $\k$ be algebraically closed of prime characteristic $p>0$. Let $H$ be any pointed Hopf algebra of dimension $p^3$ over $\k$ with $H_0= \k G= \k G(H)$. Our goal is to classify all isomorphism classes of such $H$. Our strategy is first to classify $\grH \cong R \# \k G$, where $R$ is a graded braided Hopf algebra in $\YD$, then lift $\grH$ structures to $H$. As the dimension of $H_0$ divides the dimension of $H$ by Nichols-Zoeller Theorem \cite{NZ}, we consider the following cases: 
 
(a) If $\dim H_0=1$, then $H_0=\k$, i.e., $H$ is connected. This part was done by joint work of both authors with L.~Wang in \cite{NWW1,NWW2}. 

(b) If $\dim H_0=p$, then $G=C_p=\langle g\rangle$ and $\dim R=p^2$. This part is discussed in Section~\ref{subsec:prim} cases (A)-(B) and in Section~\ref{subsec:non-prim} case (C), where (C) is the non-primitively generated case.

(c) If $\dim H_0=p^2$, then $G=C_{p^2}$ or $C_p\times C_p$ and $\dim R=p$. This part is handled in Section~\ref{subsec:prim} case (D).

(d) If $\dim H_0=p^3$, then $H=\k G$ is a group algebra of dimension $p^3$. The classification of finite groups of order $p^3$ is well-known:
\begin{itemize}
\item \textbf{abelian}: $C_{p^3}, C_{p^2}\times C_p, C_p\times C_p\times C_p$.
\item \textbf{non-abelian}: dihedral group $D_4$ and quaternion group $Q_8$ for $p=2$; and for $p>2$, Heisenberg group $UT(3,p)$ and 
\[
G_p=\left\{
\begin{pmatrix}
1+(p)  &  b\\
0  &       1
\end{pmatrix},
(p)=p\mathbb Z/(p^2),\ \text{and}\ b\in \mathbb Z/(p^2)
\right\}.
\]
\end{itemize}


\subsection{$R$ is primitively generated} 
\label{subsec:prim}

In this section, suppose that $R$ is primitively generated, that is $R$ is generated by $R_0=\k$ and $R_1=\Prim(R)$ of primitive elements. In this case, we know $R=\B(R_1)$ Nichols algebra, which is uniquely determined by the infinitesimal braiding of $R_1$.

If $H_0=\k C_p$ with generator $g$ of $C_p$, then $R_1$ is two-dimensional with basis $\{r_1, r_2\}$. The structure of $R \in \YD$ depends on how $G$ (co)acts on these basis elements. We apply Lemma \ref{L:PositiveYDRank2}.\\

\noindent
\textbf{Case (A).} $G=C_p=\langle g\rangle$ and $R_1$ is of diagonal type. Suppose $p>2$, this forces the action of $G$ to be trivial:
$$g \cdot r_1 = r_1, \qquad \qquad g \cdot r_2 = r_2.$$ 
Assume $r_1, r_2$ have $G$-gradings  $g^u, g^v$, respectively, for some $u, v \in \{0, 1, \ldots, p-1\}$. Without loss of generality, we may obtain either
\begin{itemize}
 \item \textbf{(A1)}: The $G$-coaction is non-trivial. We can assume $\rho_R(r_1)=g\otimes r_1$ and $\rho_R(r_2)=g^u \otimes r_2$, for $0\le u\le p-1$; or
 \item \textbf{(A2)}: The $G$-coaction is trivial, that is $\rho_R(r_1)=1\otimes r_1$ and $\rho_R(r_2)=1 \otimes r_2$.
\end{itemize} 
Suppose $p=2$, then the $G$-coaction is trivial and $g$ interchanges $r_1$ and $r_2$ such that  
\begin{itemize}
 \item \textbf{(A3)}: $g \cdot r_1 = r_2, \, g \cdot r_2 = r_1$, and $\rho_R(r_1)=1\otimes r_1, \, \rho_R(r_2)=1 \otimes r_2$.
\end{itemize} 
By Proposition~\ref{NicholsAlgebraRank2}, the structure of $R$ is determined by $R=\k[r_1,r_2]/(r_1^p,r_2^p)$ in all three cases. \\

\noindent
\textbf{Case (B).} $G=C_p=\langle g\rangle$ and $R_1$ is of Jordan type. This only happens when $p>2$, since if $p=2$ Proposition~\ref{NicholsAlgebraRank2} shows that $\dim B(R_1)=2^4\neq 2^2$. We can assume $r_1, r_2$ have the same $G$-grading $g$ with $G$-action 
$$g \cdot r_1 = r_1\quad g \cdot r_2 = r_1 + r_2.$$ 
By Proposition~\ref{NicholsAlgebraRank2}, the structure of $R$ is determined by $R=\k\langle r_1,r_2\rangle /(r_1^p,r_2^p,r_1r_2-r_2r_1-\frac{1}{2}r_1^2)$.\\
 
\noindent
\textbf{Case (D).} $G=C_{p^2}$ or $C_p\times C_p$. Hence $R_1$ is one-dimensional with basis element $r$ of some $G$-grading $g_u \in G$. Clearly $G$ trivially acts on $r$. 

If $G=C_{p^2}$, a cyclic group generated by $g$, then $g_u = g^s$ for some $s \geq 0$. There are three cases that can happen for this $G$-grading $g^s$: 
\begin{itemize}
 \item $s=0$, hence $g^s = 1$; 
 \item $s$ is coprime with $p^2$, then by changing generator, we can rewrite $g^s$ as $g$;  
 \item $s=pv$, where $v$ is coprime with $p$, then by changing generator, we can rewrite $g^s$ as $g^p$. 
\end{itemize}
It is clear that these three cases are not isomorphic. 

If $G = C_p\times C_p =\langle g_1 \rangle\times \langle g_2 \rangle$, one can view $G$ as a two-dimensional vector space over $\mathbb{F}_p$, picking a generator of $G$ is the same as picking a basis of this vector space. Thus, having $G$-grading $g_u=(g_1,1) \in G$ yields isomorphic result as having $G$-grading $g_u=(1,g_2)$.

Therefore, in case (D), without loss of generality, we can obtain either
\begin{itemize}
 \item \textbf{(D1)}: When $G=C_{p^2}=\langle g\rangle$, we can take $g_u=1$, $g$ and $g^p$; or 
 \item \textbf{(D2)}: When $G=C_p\times C_p=\langle g_1 \rangle\times \langle g_2 \rangle$, we can take $g_u=1$ or $g_1$.
\end{itemize} 
It is easy to see that $R$ is determined by $R=\k[r]/(r^p)$.


\subsection{$R$ is non-primitively generated} 
\label{subsec:non-prim}
In positive characteristic, $R$ is not primitively generated in general. Thus, we need to consider an additional case: \\

\noindent
\textbf{Case (C).} When $G=C_p=\langle g\rangle$ and assume $R$ is generated by a primitive element $r \in R_1$ and a \textit{non-primitive} element $z$. Then the only way $G$ can act on primitive generator $r$ is 
$$g \cdot r = r, \qquad \qquad \rho_R(r) = g^u \otimes r,$$ 
for some $u \in \{0, 1, \ldots, p-1\}$. Without loss of generality, we can assume $g^u=1$ or $g$; and under this $G$-action, we have $r^p=0$. Let $\B(R_1) = \k[r]/(r^p)$ be the Nichols algebra of $R_1$. Then $\B(R_1)$ is a graded braided Hopf subalgebra of $R$ of dimension $p$. \\

\noindent
\textbf{Claim:} $\dim\HL^2(\B(R_1),\k)=1$ and it is spanned by the 2-cocycle
\begin{align}
\label{E:Bock}
\Bock(r):=\sum_{1\le i\le p-1} \frac{(p-1)!}{i!\,(p-i)!}\, \left(r^i\otimes r^{p-i}\right). 
\end{align}
\begin{proof}[Proof of claim]
Let $R_1^*$ be the dual of $R_1$ in $\YD$. It is clear that $R_1^*$ is still one-dimensional, hence it has trivial braiding since the $G$-action must be trivial. So one sees that $\B(R_1^*)\cong \k C_p$ as algebras. By Lemma \ref{Cobar}, we have 
\[
\dim \HL^2(\B(R_1),\k)=\dim \HH^2(\B(R_1^*),\k)=\dim \HH^2(\k C_p,\k)=\dim \HL^2(C_p,\k)=1.
\] 
It is easy to check that $\Bock(r)$ is a $2$-cocycle in the cobar construction $\Omega\B(R_1)$. By a similar degree argument as in \cite[Proposition 6.2]{wang2012connected}, one can conclude that $\Bock(r)$ does not lie in the coboundary of $\Omega\B(R_1)$.
\end{proof}

Observe that $\Bock(r)$ has total degree (Adams grading) $p$ in $\B(R_1)$ and $\B(R_1)_p=0$. We apply \Cref{ExtNichols} to get an injective Yetter-Drinfeld module map $\partial^1: R_p\hookrightarrow \HH^{2,p}(\B(R_1),\k)$. By a simple dimension argument, one sees that $\partial^1$ is an isomorphism. Thus we can choose the non-primitive element $z$ to be a basis of $R_p$ and write 
\[\Delta(z)= z \otimes 1 + 1 \otimes z + \Bock(r).\] 
Moreover, one checks that $ \Bock(r)$ has trivial $G$-grading since $g^{up}=1$ by (\ref{E:Bock}). Hence $z$ also has trivial $G$-grading such that $\rho_R(z) = 1 \otimes z$. 

To determine the relations between $r$ and $z$, observe that the $G$-actions on them are both trivial. Hence 
\begin{align*}
\Delta\left([r,z]\right)&\, =\Delta\left(rz-zr\right)=[r\otimes 1+1\otimes r,z\otimes 1+1\otimes z+\Bock(r)]\\
&\, =(rz - zr)\otimes 1+ 1\otimes (rz-zr)=[r,z]\otimes 1+1\otimes [r,z].
\end{align*}
So $[r,z] \in R_1\cap R_{p+1}=0$ and $rz=zr$. One can also easily compute $\Delta(z^p) = (z \otimes 1 + 1 \otimes z + \Bock(r))^p = z^p \otimes 1 + 1 \otimes z^p + \Bock(r^p) = z^p\otimes 1+1\otimes z^p$, implying $z^p \in R_1\cap R_{p^2}=0$. 

As a conclusion, for case (C), the structure of $R$ is given by 
$$R = \k[r, z] / (r^p, z^p),$$ 
where $r \in R_1$ is a primitive element of $G$-grading $1$ or $g$, and $z \in R_p$ is a non-primitive element of trivial $G$-grading.

In summary, we obtain the following $R$-structures in $\grH = R \# H_0$ in \Cref{tab:typeR}. Here, we use use $g$ to denote the generators of the cyclic groups $C_p$ and $C_{p^2}$; use $g_1,g_2$ to denote the generators of $C_p\times C_p$.  The notation $\epsilon \in \{0,1,p\}$ is used to mean $g^\epsilon=1$ or $g$ in $C_p$, and $g^\epsilon=1,g,$ or $g^p$ in $C_{p^2}$.

\FloatBarrier
\begin{table}[!htp]
\begin{center}
\caption{Structure of $R$ -- for pointed $p^3$-dimensional Hopf algebras in characteristic $p>0$, when dim $H_0=p$ or $p^2$}\label{tab:typeR} \vspace{-0.2cm}
\begin{tabular}{| p{0.95cm} | p{1.35cm} | p{5cm} | p{4.6cm} | p{1.4cm} |} 

\hline
\bf{Type of $R$} & $G(H)$ & \bf{(Co)Alg. Structure of $R$}  & \bf{$\YD$-mod Structure of $R$} & \bf{\# Iso. classes of $R$} \\
\hline   \hline  
(A1)  & $C_p$ & $r_1^p=r_2^p=0, r_1r_2=r_2r_1,\ \Delta(r_i)=r_i \otimes 1+1 \otimes r_i$ & $g \cdot r_i = r_i,\ \rho(r_1) = g \otimes r_1,\ \rho(r_2) = g^{u} \otimes r_2, 0\le u\le p-1$  & $p$ classes \\  \hline

(A2)  & $C_p$ & $r_1^p=r_2^p=0, r_1r_2=r_2r_1,\ \Delta(r_i)=r_i \otimes 1+1 \otimes r_i$,  & $g \cdot r_i = r_i,\ \rho(r_i) = 1\otimes r_i$  & one class \\  \hline

(A3) {\tiny $p=2$}   & $C_p$ & $r_1^p=r_2^p=0, r_1r_2=r_2r_1,\ \Delta(r_i)=r_i \otimes 1+1 \otimes r_i$,  & $g \cdot r_1= r_2, g\cdot r_2=r_1,\ \rho(r_i) = 1\otimes r_i$  & one class \\  \hline

(B) {\tiny $p>2$} & $C_p$ & $r_1^p=r_2^p=0, r_1r_2-r_2r_1=\frac{1}{2}r_1^2,\ \Delta(r_i)=r_i \otimes 1+1 \otimes r_i$ & $g \cdot r_1 = r_1, g \cdot r_2 = r_2 + r_1,\ \rho(r_i) = g \otimes r_i$  & one class \\  \hline

(C)  & $C_p$ & $r^p=z^p=0, rz=zr, \Delta(r)=r \otimes 1+1 \otimes r,\ \Delta(z)=z \otimes 1+1 \otimes z+\Bock(r)$ & $g \cdot r = r, g \cdot z = z,\ \rho(r) = g^{\epsilon}\otimes r, \rho(z) = 1 \otimes z$  & two classes \\  \hline

(D1)  & $C_{p^2}$ & $r^p=0,\ \Delta(r)=r \otimes 1+1 \otimes r$ & $g \cdot r = r,\ \rho(r) = g^{\epsilon} \otimes r$ & three classes \\  \hline

(D2)  & $C_p\times C_p$ & $r^p=0,\ \Delta(r)=1 \otimes r + r \otimes 1$ & $g_i \cdot r = r,\ \rho(r) = g_1^{\epsilon} \otimes r$ &  two classes \\ \hline
\end{tabular}
\end{center}
\end{table}


\subsection{Structure of $\grH$}
\label{subsec:grH}

Given $R$ and $H_0$ as classified in the previous section, here we will obtain their bosonization $\grH \cong R \# H_0$. In the next section, we will lift the structures of $\grH$ to $H$ to classify pointed $p^3$-dimensional Hopf algebras in positive characteristic $p>0$. 

Let us first set up some conventions. The underlying vector space structure of $\grH \cong R \# H_0$ is $R \otimes H_0$. For all $g \in G$, we identify $g=1 \# g \in \grH$ via the inclusion $\iota$. For cases (A1)-(A3) and (B), let $a:=r_1 \# 1$ and $b:=r_2 \# 1$. For case (C), let $a:=r \# 1$ and $b:=z \# 1$. For cases (D1)-(D2), let $a:=r \# 1$. 

We compute the structure of $\grH$ using the structure formulas as seen in c.f.~Section~\ref{sec:prelim}, \cite{R85}, or \cite[Theorem 10.6.5]{MO93}. For cases (A1)-(A3), (B), and (D), when $R$ is primitively generated, one can compute the bosonization structure of $\grH \cong \BV \# \k G$ using simpler formulas, e.g. for all $g \in G$ and $v \in V=R_1$:
$$gv = (g \cdot v) \# g, \quad \text{ and} \quad \Delta(v)= v \otimes 1 + v_{(-1)} \otimes v_{(0)}.$$

Since the computation is straightforward, we do not give details here. We obtain the following bosonization structures for $\grH \cong R \# H_0$ in Table~\ref{tab:grH}. It is a $\k$-algebra generated by $\{a, b, g\}$ in cases (A1)-(A3), (B), and (C); generated by $\{a,g\}$ in case (D1); and generated by $\{a,g_1,g_2\}$ in case (D2), where $C_p\times C_p = \langle g_1 \rangle \times \langle g_2 \rangle$. We also note that the Hopf algebras $\grH$ of case (B) have appeared in \cite{CLW}.

\FloatBarrier
\begin{table}[!htp]
\begin{center}
\caption{Structure of $\grH$ -- for pointed $p^3$-dimensional Hopf algebras in characteristic $p>0$, when dim $H_0=p$ or $p^2$}\label{tab:grH} \vspace{-0.2cm}
\begin{tabular}{| p{0.9cm} | p{0.8cm} | p{3.05cm} | p{7.1cm} | p{1.4cm} |} 
\hline
\bf{Type of $R$} & $G(H)$ & \bf{Algebra Structure of $\grH$}  & \bf{Hopf Structure of $\grH$} & \bf{\# Iso. classes of $\grH$} \\
\hline   \hline  
(A1)  & $C_p$ & $a^p=b^p=0, g^p = 1, ab=ba, \,ga = ag, \,gb=bg$ & $\Delta(g)= g \otimes g, \Delta(a)=a \otimes 1+g \otimes a, \Delta(b)=b \otimes 1+g^u\otimes b, \varepsilon(g)=1, \varepsilon(a)=\varepsilon(b)=0, S(g)=g^{-1}, S(a) = -ag^{-1}, S(b)= -bg^{-u}, 0\le u\le p-1$  & $p$ classes  \\   \hline

(A2)  & $C_p$ & $a^p=b^p=0, g^p = 1, ab=ba, \,ga = ag, \,gb=bg$ & $\Delta(g)= g \otimes g, \Delta(a)= a \otimes 1+1 \otimes a, \Delta(b)=b \otimes 1+1 \otimes b, \varepsilon(g)=1, \varepsilon(a)=\varepsilon(b)=0, S(g)=g^{-1}, S(a) = -a, S(b)= -b$   & one class  \\  \hline

(A3) {\tiny $p=2$}  & $C_p$ & $a^p=b^p=0, g^p = 1, ab=ba, \,ga = bg, \,gb=ag$ & $\Delta(g)= g \otimes g, \Delta(a)= a \otimes 1+1 \otimes a, \Delta(b)=b \otimes 1+1 \otimes b, \varepsilon(g)=1, \varepsilon(a)=\varepsilon(b)=0, S(g)=g^{-1}, S(a) = -a, S(b)= -b$   & one class  \\  \hline

(B) {\tiny $p>2$}  & $C_p$ & $a^p=b^p=0,\,  g^p = 1,\,  ab-ba=\frac{1}{2}a^2, \,ga=ag, \,gb=(a+b)g$ & $\Delta(g)= g \otimes g, \Delta(a)=a \otimes 1+g\otimes a, \Delta(b)=b \otimes 1+g\otimes b, \varepsilon(g)=1, \varepsilon(a)=\varepsilon(b)=0, S(g)=g^{-1}, S(a) = -ag^{-1}, S(b)=(a-b)g^{-1}$  & one class \\  \hline

(C)  & $C_p$ & $a^p=b^p=0, g^p = 1, ab=ba, \,ga = ag, \,gb=bg$ & $\Delta(g)= g \otimes g,\, \Delta(a)=a \otimes 1+g^\epsilon \otimes a,\, \Delta(b)=b \otimes 1 +1 \otimes b + \sum_{1 \le i \le p-1} \frac{(p-1)!}{i!\,(p-i)!}\, (a^i \, g^{\epsilon(p-i)} \otimes a^{p-i} ),\varepsilon(g)=1, \varepsilon(a)=\varepsilon(b)=0, S(g)=g^{-1}, S(a) = -ag^{-\epsilon}, S(b)= -b$   & two classes \\  \hline

(D1)  & $C_{p^2}$ & $a^p=0, g^{p^2} = 1, \,ga = ag$ & $\Delta(g)= g \otimes g, \Delta(a)=a \otimes 1+g^{\epsilon} \otimes a , \varepsilon(g)=1, \varepsilon(a)=0, S(g)=g^{-1}, S(a) = -ag^{-\epsilon}$  & three classes \\  \hline

(D2)  & $C_p\times C_p$ & $a^p=0, g_1^{p} = g_2^{p}=1, \,g_ia = ag_i$ & $\Delta(g_i)= g_i \otimes g_i, \Delta(a)=a \otimes 1+g_1^\epsilon \otimes a, \varepsilon(g_i)=1, \varepsilon(a)=0, S(g_i)=g_i^{-1}, S(a) = -ag_1^{-\epsilon}$ & two classes \\ \hline
\end{tabular}
\end{center}
\end{table}


\section{Pointed $p^3$-dimensional Hopf algebras: liftings from $\grH$ to $H$}
\label{sec:pointed p3}

We will now lift the structures of $\grH$, as seen in Table~\ref{tab:grH}, to $H$. To do that, we need the following two lemmas:

\begin{lem}
\label{pthpower}
Let $G$ be a finite $p$-group and $\delta: \k G\to \k G$ be a $\k$-derivation on $\k G$. Then
\begin{enumerate}
\item Suppose $G=\langle g\rangle$ is cyclic of order $p^n=:q$ and $\delta(g)=g-g^2$. Then $\delta^m(g)=\sum_{0\le i\le q-1} a^m_ig^i$, for $m\ge 1$, where the coefficients are defined by
$$
a^m_0=\sum_{1\le j\le q-1}(-1)^j{q-1\choose j}j^m,\qquad a^m_i=\sum_{0\le j\le i-1} (-1)^j {i-1\choose j}(j+1)^m,
$$
for $1\le i\le q-1$. In particular, we have $\delta^p=\delta$ and $\delta^{p-1}(g)=g-g^p$.

\item Suppose $G=\langle g\rangle$ is cyclic of order $p$ and $\delta(g)=g-g^2$. For any integer $1\le u\le p-1$, we have $$\left(\frac{\delta}{u} + \frac{\delta^2}{u^2} + \cdots + \frac{\delta^{p-1}}{u^{p-1}}\right)(g^{1+u})=0.$$

\item Suppose $G=C_p$ is cyclic of order $p$ and $\delta(g)=g-g^{u+1}$, for any $0\le u\le p-1$. Then we have $\delta^p=\delta$.

\item Suppose $G=C_p\times C_p=\langle g\rangle\times \langle h\rangle $ and $\delta(g)=g-g^2$, $\delta(h)=\tau h(1-g)$, for some $\tau\in \k$. Then $\delta^p=\delta$.

\end{enumerate}
\end{lem}

\begin{proof}[Proof of Lemma~\ref{pthpower}]
(1) We consider the matrix $M$ of the $\k$-linear map $\delta$ under the basis $\{1=g^0,\, g,\, g^2,\, \ldots,\, g^{q-1}\}$ of $\k G$, where we associate each $g^i$ to a standard column vector $e_i = [0, \ldots, 0, 1, 0, \ldots, 0]^T$, with the $(i+1)$-th entry is $1$ and zero everywhere else.
{\small
\[ M= \left(%
\begin{array}{cccccccc}
 0& 0 &  0 &  0 &\ldots &0 &0 &-(q-1) \\
 0 & 1 & 0 & 0 &\ldots &0&0 &0 \\
 0 & -1 & 2 & 0 &\ldots &0 &0 &0   \\
 0 & 0 & -2 & 3 &\ldots &0&0 &0   \\
 \vdots & \vdots & \vdots & \vdots  & \ddots &\vdots &\vdots &\vdots \\
 0 & 0 & 0 & 0 &\ldots &  -(q-3) &(q-2)& 0 \\
 0 & 0 & 0 & 0 & \ldots &0 & -(q-2) & (q-1) \\
\end{array}%
\right)_{q\times q}.
\]
}

The matrix $M$ can be diagonalized as
{\small
\[PMP^{-1} =
\left(
 \begin{array}{ccccc}
   0  & \ &\  &  \\
    & 1 &\  & \ \\
    &\ & 2\\
    &  & & \ddots\\
    & & & & (q-1)
 \end{array}
\right)_{q\times q}=:D\ \text{, by}\]
}
{\small
\begin{align*}
P &= \left(%
\begin{array}{ccccccccc}
 1 &  1 & 1 & 1 &1 & \ldots & 1 &1  \\
 0 & 1 & 0 &0 &0 &\ldots &0 &0\\
 0 &-1 & 1 & 0 &0 &\ldots &0 & 0  \\
 0&1 & -2 & 1 &0 &\ldots & 0 & 0\\
 0&-1 & 3 & -3 &1 &\ldots & 0 & 0  \\
 \vdots & \vdots & \vdots & \vdots &\vdots &\ddots &\vdots &\vdots   \\
0 & 1 & -\binom{q-3}{1} &\binom{q-3}{2} & -\binom{q-3}{3}&\ldots &1 & 0\\
0& -1 & \binom{q-2}{1} &- \binom{q-2}{2}  & \binom{q-2}{3} & \ldots & -\binom{q-2}{q-3}  & 1 \\
\end{array}%
\right)_{q\times q}\quad \text{and} \\ 
P^{-1}&= \left(%
\begin{array}{ccccccccc}
 1 &  -\binom{q-1}{1} & -\binom{q-1}{2} & -\binom{q-1}{3} &-\binom{q-1}{4} & \ldots & -\binom{q-1}{q-2} &-1  \\
 0 & 1 & 0 &0 &0 &\ldots &0 &0\\
 0 &1 & 1 & 0 &0 &\ldots &0 & 0  \\
 0&1 & 2 & 1 &0 &\ldots & 0 & 0\\
 0&1 & 3 & 3 &1 &\ldots & 0 & 0  \\
 \vdots & \vdots & \vdots & \vdots &\vdots &\ddots &\vdots &\vdots   \\
0 &1 & \binom{q-3}{1} &\binom{q-3}{2} & \binom{q-3}{3}&\ldots &1 & 0\\
0& 1 & \binom{q-2}{1} & \binom{q-2}{2}  & \binom{q-2}{3} & \ldots & \binom{q-2}{q-3}  & 1 \\
\end{array}
\right)_{q\times q}.
\end{align*}
}
Therefore, applying on $g =e_1$, we have: 
\begin{align*}
\delta^m(g)=&\, M^m(e_1)=(P^{-1}DP)^m(e_1)=P^{-1}D^mP(e_1)\\
=&\,P^{-1}D^m[1,1,-1,\dots, 1,-1]^T\\
=&\,P^{-1}[0,1^m,-2^m,\dots,(q-2)^m,-(q-1)^m]^T\\
=&\, \sum_{0\le i\le q-1} a^m_ie_i.
\end{align*}
Then it is easy to find the formulas for all the coefficients $a^m_i$. Moreover, we have $\delta^p=M^p=(P^{-1}DP)^p=P^{-1}D^pP=P^{-1}DP=M=\delta$. For $\delta^{p-1}$, by the definition of $\delta$, one sees that $a_i^{p-1}=0$, for all $i>p$. If $q=p$, then it is easy to check that $a^{p-1}_0=-1$, $a^{p-1}_1=1$ and $a^{p-1}_i=0$, for $2\le i\le p-1$. Hence $\delta^{p-1}(g)=g-1=g-g^p$. The argument for $q>p$ is similar. \\

(2) By the proof of (1), there are eigenvectors $v_0,v_1,\dots, v_{p-1}$ in $\k G$ such that $\delta(v_i)=iv_i$, for $0\le i\le p-1$. Then one sees that 
\begin{align*}
\left(\frac{\delta}{u}+(\frac{\delta}{u})^2+\cdots (\frac{\delta}{u})^{p-1}\right)(v_j)&\, =\left(\prod_{2\le i\le p} (\frac{\delta}{u}-i)\right)(v_j)=\left(\prod_{2\le i\le p} (\frac{j}{u}-i)\right)(v_j)\\
&\, =
\begin{cases}
0,   &  j\neq u\\
\left(\prod_{2\le i\le p} (1-i)\right)(v_j)=(p-1)! \, (v_j), &   j=u
\end{cases}
\end{align*}
Moreover, by using the matrix $P^{-1}$, $v_i=P^{-1}(e_i)$. Hence 
\[
g^{1+u}=e_{1+u}=
\begin{cases}
\sum_{0\le i\le p-1}v_iP_{i+1,u+2}=v_0+\sum\limits_{u\le i\le p-2}(-1)^{i-u} {i \choose u}v_{i+1}, & 1\le u\le p-2\\
v_0,  &  u=p-1
\end{cases} 
\]
So when we write $g^{1+u}$ as a linear combination of $v_i$'s, it  does not contain the linear term $v_u$. Thus $(\delta/u+\delta^2/u^2+\cdots \delta^{p-1}/u^{p-1})(g^{1+u})=0$.\\

(3) As in the proof of (1), we consider the matrix $M$ of the $\k$-linear map $\delta$ under  the basis $\{1,g,g^2,\dots, g^{p-1}\}$. One sees that $M$ is diagonalizable with eigenvalues $0,1,2,\dots,p-1$. Hence $\delta^p=\delta$. 

(4) We put an ordering on the the basis $\{g^ih^j\,|\, 0\le i,j\le p-1\}$ of $\k G$ such that $g^{i_1}h^{j_1}<g^{i_2}h^{j_2}$ whenever $(j_1<j_2)$ or $(i_1<i_2$ if $j_1=j_2)$. Then we consider the matrix $M$ of the $\k$-linear map $\delta$ under the ordered basis $\{g^ih^j\,|\, 0\le i,j\le p-1\}$. Since $\delta(g^ih^j)=(j\tau+i)g^ih^j-(j\tau+i)g^{i+1}h^j$, $M=\text{Diag}(B_0,B_1,\dots,B_{p-1})$ is a block matrix such that 
{\small
\begin{align*}
B_0&= \left(%
\begin{array}{cccccccc}
 0& 0 &  0 &  0 &\ldots &0 &0 &-(p-1) \\
 0 & 1 & 0 & 0 &\ldots &0&0 &0 \\
 0 & -1 & 2 & 0 &\ldots &0 &0 &0   \\
 0 & 0 & -2 & 3 &\ldots &0&0 &0   \\
 \vdots & \vdots & \vdots & \vdots  & \ddots &\vdots &\vdots &\vdots \\
 0 & 0 & 0 & 0 &\ldots &  -(p-3) &(p-2)& 0 \\
 0 & 0 & 0 & 0 & \ldots &0 & -(p-2) & (p-1) \\
\end{array}%
\right)_{p\times p}\quad \text{and} \\
B_j&= \left(%
\begin{array}{ccccccc}
 j\tau& 0 &  0 &  0 &\ldots &0 &-j\tau-(p-1) \\
 -j\tau& j\tau+1 & 0 & 0 &\ldots &0 &0 \\
 0 & -j\tau-1 & j\tau+2 & 0 &\ldots &0 &0   \\
 0 & 0 & -j\tau-2 & j\tau+3 &\ldots &0 &0   \\
 \vdots & \vdots & \vdots & \vdots  & \ddots  &\vdots &\vdots \\
 0 & 0 & 0 & 0 &\ldots &j\tau+(p-2)& 0 \\
 0 & 0 & 0 & 0 & \ldots & -j\tau-(p-2) & j\tau+(p-1) \\
\end{array}%
\right)_{p\times p},
\end{align*}
}
for $1\le j\le p-1$. It suffices to show that each block $B_j$ has eigenvalues $0,1,\dots,p-1$, (mod $p$). Then, since $B_j$ has size $p\times p$, it is diagonalizable and $B_j^p=B_j$. So $M^p=M$, which implies that $\delta^p=\delta$. From the proof of (1), this is true for $B_0$ ($p=q$). For $B_j$ ($1\le j\le p-1$), one can compute its characteristic polynomial as follows
\begin{align*}
f_j(t)=\prod_{0\le i\le p-1}\left(t-(j\tau+i)\right)-\prod_{0\le i\le p-1}(-j\tau-i).
\end{align*}
It is clear to check that $f_j(t)$ has roots $0,1, \dots,p-1$, (mod $p$). Then by the degree argument, one sees that $f_j(t)=t^p-t$, again in mod $p$. Thus the result follows. 
\end{proof}

\begin{lem}
\label{pthcoproduct}
Let $G$ be a cyclic $p$-group generated by $g$. In the $\k$-algebra $A=\k \langle G,x\rangle$, suppose the relation $gx-xg=\mu\,(g-g^2)$ holds for some $\mu\in \k$. Then we have $(g)(\ad\, x)^{p-1}=\mu^{p-1}\, (g-g^p)$. Moreover in the tensor algebra $A\otimes A$, we have $(x\otimes 1+g\otimes x)^p=x^p\otimes 1+g^p\otimes x^p+\mu^{p-1}(g-g^p)\otimes x$. 
\end{lem}

\begin{proof}
By applying Lemma~\ref{pthpower} (1), with $\delta(g)= (g)(\ad \, \frac{x}{\mu})$, then we have $\delta^{p-1}(g) = \frac{1}{\mu^{p-1}}(g)(\ad \, x)^{p-1} = g-g^p$. Hence, $(g)(\ad\, x)^{p-1}=\mu^{p-1}\, (g-g^p)$. Now the statement holds by applying Proposition~\ref{palgebra}, $(x \otimes 1 + g \otimes x)^p = x^p \otimes 1 + g^p \otimes x^p + (g)(\ad \, x)^{p-1} \otimes x$. 

Alternatively, one may also derive from the argument in \cite[Corollary 4.10]{S} and \cite{WangWang} to show the result.
\end{proof}


Before we deal with each case in \Cref{tab:grH}, we state our strategy to determine all possible liftings from $\grH$ to $H$. Suppose $\grH=\k\langle x_1,x_2,\dots\rangle/(f_1,f_2,\dots)$ with coalgebra structure defined on $x_i$. By the structure of pointed coalgebras, one sees clearly that the generators $x_i$ of $\grH$ can be lifted up to be generators of $H$ where the coalgebra structure remains the same once $x_i\in H_1$. The case $(C)$ is the only exception since we have one generator which is not in $H_1$, whose coalgebra structure can be uniquely lifted is proved by computing the Hochschild cohomology. In the Hopf algebra $H$, the old relations $f_i$ in $\grH$ becomes $f_i=r_i$ where $r_i\in H_{n-1}$ if the homogenous relation $f_i$ has degree $n$ in $\grH$. Second, we compute the comultiplication of each $r_i=f_i$ in $H$. Based on the calculation, we explicitly construct some term $a_i$ to make $r_i-a_i$ skew-primitive, or $r_i-a_i\in H_1$. Next we can write $r_i-a_i$ as a linear combination of group like elements and skew-primitive elements of $H$ with suitable coefficients. Finally, we need to check that all the coefficients are compatible in the way that the whole algebra modulo these relations is of dimension $p^3$. 

In the following subsections, for completeness, we provide as much lifting details as possible. The readers may check closely for the argument of one case and the other cases should follow similarly, possibly with additional complications arising from the structure of each case. 


\subsection{Liftings for Case (A)}
\label{liftingA}

Following the above plan and applying Lemmas \ref{pthpower} and \ref{pthcoproduct} frequently in our calculations, we will do the liftings for each case. \\

\noindent 
\textbf{Case (A1).} It is clear that $H$ is the quotient of the free algebra $\k\langle g,x,y\rangle$, where $x,y \in H$ are liftings of $a,b \in \grH$, respectively, subject to the relations 
\[
g^p=1,\ gx-xg=r_1,\ gy-yg=r_2,\ x^p=r_3,\ y^p=r_4,\ xy-yx=r_5,
\] 
for some $r_1,r_2\in H_0$, $r_5\in H_1$ and $r_3,r_4\in H_{p-1}$. The coalgebra structure is determined by  
\[\Delta(g)=g\otimes g,\  \Delta(x)=x\otimes 1+g\otimes x,\ \Delta(y)=y\otimes 1+g^u\otimes y,\]
for some $0\le u\le p-1$. We first determine $r_1$ and $r_2$.
\begin{align*}
\Delta(r_1)&\, =\Delta(gx-xg)=(g\otimes g)(x\otimes 1+g\otimes x)-(x\otimes 1+g\otimes x)(g\otimes g)\\
&\, =(gx-xg)\otimes g+g^2\otimes (gx-xg)=r_1\otimes g+g^2\otimes r_1.
\end{align*}
Since $r_1\in H_0=\k C_p$. It is easy to see that we can write $r_1=\gamma(g-g^2)$, for some $\gamma \in \k$. By a suitable rescaling of the variable $x$, we can take $\gamma \in \{0,1\}$. In the following, we use the notation $r_1=\epsilon_1\, g(1-g)$, where $\epsilon_1 \in \{0,1\}$. Similarly, we have $r_2=\epsilon_2\, g(1-g^u)$, for $\epsilon_2 \in \{0,1\}$.

Next, we apply \Cref{pthcoproduct} to get
\begin{align*}
\Delta(r_3)=\Delta(x^p)=(x\otimes 1+g\otimes x)^p=x^p\otimes 1+1\otimes x^p+\epsilon_1(g-1)\otimes x
\end{align*}
since $g^p=1$, and $\epsilon_1^n=\epsilon_1$, for any $n \geq 0$ and $\epsilon_1 \in \{0,1\}$ in this case. So
\begin{align*}
\Delta(r_3-\epsilon_1\, x)&\, =x^p\otimes 1+1\otimes x^p+\epsilon_1(g\otimes x-1\otimes x)-\epsilon_1(x\otimes 1+g\otimes x)\\
&\, =(x^p-\epsilon_1\, x)\otimes 1+1\otimes (x^p-\epsilon_1\, x).
\end{align*}
Hence, $(r_3-\epsilon_1\, x)$ is primitive. For $r_4$, we use the fact that $[-,y]$ is a derivation, so $[g^u,y]=ug^{u-1}[g,y]=u\epsilon_2\, g^u(1-g^u)$. Hence, by applying \Cref{pthcoproduct} again, we obtain $(r_4-u^{p-1}\epsilon_2\, y)$ is primitive. 

Moreover, we get
\begin{align*}
\Delta(r_5)&\, =\Delta(xy-yx) \\
&\,=(x\otimes 1+g\otimes x)(y\otimes 1+g^u\otimes y)-(y\otimes 1+g^u\otimes y)(x\otimes 1+g\otimes x)\\
&\, =(xy-yx)\otimes 1+g^{u+1}\otimes (xy-yx)+(xg^u-g^ux)\otimes y+(gy-yg)\otimes x\\
&\, =r_5\otimes 1+g^{u+1}\otimes r_5-u\epsilon_1\, (g^u-g^{u+1})\otimes y+\epsilon_2(g-g^{u+1})\otimes x.
\end{align*}
Then it is easy to see that 
\[
\Delta(r_5+u\epsilon_1\, y-\epsilon_2\, x)=(r_5+u\epsilon_1\, y-\epsilon_2\, x)\otimes 1+g^{u+1}\otimes (r_5+u\epsilon_1\, y-\epsilon_2\, x).
\] 
Finally, we point out that if $u\neq 0$, then the primitive space $P(H)=0$. Otherwise, if $u=0$, then $P(H)=\k y$ and $\Delta(r_5)=r_5\otimes 1+g\otimes r_5$. Therefore, we obtain the following two cases depending on the value of $u$.

\begin{itemize}
\item \textbf{(A1a)} Suppose $u=0$. Then $H$ is the quotient of the free algebra $\k\langle g,x,y\rangle$ modulo the possible relations 
\begin{align*}
g^p=1, \ gx-xg&=\epsilon_1\, (g-g^2),\ gy=yg,\ x^p=\epsilon_1\, x+\lambda\, y,\\ 
y^p&=\mu\, y,\ xy-yx=\sigma\, x+\tau\, (1-g),
\end{align*}
for some $\epsilon_1 \in \{0,1\}$ and some $\lambda,\mu,\sigma,\tau\in \k$. By rescaling of $y$, we can let $\mu=\epsilon_2\in \{0,1\}$. Since $H$ is of dimension $p^3$, the Diamond Lemma \cite{Beg} implies that all the ambiguities of the relations are resolvable. This can be done by comparing coefficients of each of the following expressions: 
$$[x,x^p], \ [y,x^p], \ [g,x^p], \ [x,y^p], \ [y,y^p], \ [g,y^p], \ [x,g^p], \ [y,g^p],$$
where each expression is computed in two ways: (i) by replacing the above $x^p, y^p, g^p$ relations in the commutator, and (ii) by using Proposition~~\ref{palgebra}, e.g. $[y,x^p]=(y)(\ad \, x)^p$. Then we get the compatible conditions for the coefficients:
\[
\epsilon_1 \sigma=\lambda \sigma=\lambda \tau=(\epsilon_2 - \sigma^{p-1})\sigma=(\epsilon_2 - \sigma^{p-1})\tau=0.
\]
We may interpret these conditions further: 
\begin{itemize}
\item Suppose $\epsilon_1=\epsilon_2=0$. Then we have $\sigma=0$ and $\lambda\tau=0$. By rescaling $x,y$, we can take $\lambda,\tau\in\{0,1\}$ satisfying $\lambda\tau=0$. There are totally three classes.
\item Suppose $\epsilon_1=0,\epsilon_2=1$. Then we have $\lambda\sigma=\lambda\tau=0$ and $\sigma^p=\sigma,\tau=\sigma^{p-1}\tau$. By rescaling $x,y$, we can take $\sigma, \lambda,\tau\in\{0,1\}$ satisfying $\lambda\sigma=\tau(1-\sigma)=0$. There are totally four classes.
\item Suppose $\epsilon_1=1,\epsilon_2=0$. Then we have $\sigma=\lambda\tau=0$. By rescaling $x,y$, we can take $\lambda,\tau\in\{0,1\}$ satisfying $\lambda\tau=0$. There are totally three classes.
\item Suppose $\epsilon_1=\epsilon_2=1$. Then we have $\sigma=\tau=0$. We get one infinite family of $H$ whose structures depending on $\lambda \in \k$.\\
\end{itemize}

\item \textbf{(A1b)}  Suppose $u\neq 0$. Then $H$ is the quotient of the free algebra $\k\langle g,x,y\rangle$ modulo the possible relations 
\begin{align*}
g^p=1, \ gx-xg&=\epsilon_1\, (g-g^2),\ gy-yg=\epsilon_2\,(g-g^{u+1}),\ x^p=\epsilon_1\, x,\\  
y^p&=\epsilon_2\, y,\ xy-yx+u\epsilon_1\, y-\epsilon_2\, x=\tau\, (1-g^{u+1}),
\end{align*}
for some $\epsilon_i \in \{0,1\}$ and some $\tau\in \k$. 

When we apply the Diamond Lemma \cite{Beg}, all the ambiguities of the relations are resolvable and there is no ambiguity condition in this case. Here we will check for $[g,y^p]$, $[x,y^p]$ and $[y,x^p]$ and leave the rest to the readers. For $[g,y^p]$, take $\delta=\ad\ y$ and $\delta(g)=\epsilon_2(g-g^{u+1})$. One checks that either (i) $[g,y^p]=[g,\epsilon_2y]=\epsilon_2[g,y]$, or (ii) $[g,y^p]=(g)(\ad\ y)^p=\delta^p(g)=\delta (g)=[g,y]$ by Lemma \ref{pthpower} (3). Then both ways are the same since $\epsilon_2\in \{0,1\}$. For $[x,y^p]$ and $[y,x^p]$, we need the following result.\\

\noindent
\textbf{Claim:} $\sum_{i=0}^{p-2} (u\epsilon_1)^i\, (g^{u+1})(\ad\ x)^{p-1-i}=\sum_{i=0}^{p-2} \epsilon_2^i\, (g^{u+1})(\ad\ y)^{p-1-i}=0$.

\begin{proof}[Proof of claim]
Without loss of generality, we assume $\epsilon_1=\epsilon_2=1$. For the first summation, take $\delta=\ad\ x$ and $\delta(g)=g-g^2$. Thus, by Fermat's Little Theorem, for $u \neq 0$, $u^{p-1} = 1$ in mod $p$, and we have:
\begin{align*}
\sum_{i=0}^{p-2} (u\epsilon_1)^i\, (g^{u+1})(\ad\ x)^{p-1-i}&\, = \left(\sum_{i=0}^{p-2} \frac{u^i}{u^{p-1}}\, \delta^{p-1-i} \right) (g^{u+1}) \\
&\,= \left(\sum_{i=0}^{p-2} \delta^{p-1-i}/u^{p-1-i} \right) (g^{u+1})\\
&\,=(\delta/u+\delta^2/u^2+\cdots+\delta^{p-1}/u^{p-1})(g^{u+1})=0
\end{align*}
by \Cref{pthpower} (2). For the second summation, take $\delta=(\ad\ y)/u$ and $\delta(g^u)=g^{u-1}(g-g^{u+1})=g^u(1-g^u)$. For simplicity, we change the generator of the cyclic group $G=C_p$ from $g$ to $h=g^u$.  So $\delta(h)=h-h^2$. Thus 
\begin{align*}
\sum_{i=0}^{p-2} (g^{u+1})(\ad\ y)^{p-1-i}& \,= \left(\sum_{i=0}^{p-2} (\delta u)^{p-1-i} \right)(h^{1+u^{-1}}) \\
& \,=(\delta u+\delta^2u^2+\cdots +\delta^{p-1}u^{p-1})(h^{1+u^{-1}})=0.
\end{align*}
by \Cref{pthpower} (2) again.
\end{proof}

Moreover, by induction on $n$, one can prove that for all $n\ge 1$:
\begin{align*}
(y)(\ad\ x)^n &\, =(u\epsilon_1)^{n-1}[y,x]+\tau\left(\sum_{i=0}^{n-2} (u\epsilon_1)^i\, (g^{u+1})(\ad\ x)^{n-1-i}\right),\, \text{and}\\
(x) (\ad\ y)^n&\, =\epsilon_2^{n-1}[x,y]-\tau\left(\sum_{i=0}^{n-2} \epsilon_2^i\, (g^{u+1})(\ad\ y)^{n-1-i}\right).
\end{align*}
Hence, one gets either (i) $[x,y^p]=[x,\epsilon_2 y]=\epsilon_2[x,y]$, or (ii) $[x,y^p]=x (\ad\ y)^p=\epsilon_2^{p-1}[x,y]-\tau(\sum_{i=0}^{p-2} \epsilon_2^i\, (g^{u+1})(\ad\ y)^{p-1-i})=\epsilon_2^{p-1}[x,y]$ by the claim. So both ways are the same. The expression $[y,x^p]$ can be checked similarly. 

In conclusion, if one of the $\epsilon_i$'s is zero, we can always rescale $x$ or $y$ to make $\tau\in \{0,1\}$, where, depending on $u\in \{1,2,\dots,p-1\}$, there are totally $6(p-1)$ classes. When $\epsilon_1=\epsilon_2=1$, we get $(p-1)$ infinite families of $H$ whose structures depending on $\tau\in \k$ and $u\in \{1,2,\dots,p-1\}$. \\
\end{itemize}

\noindent 
\textbf{Case (A2).} It is clear that $H$ is the quotient of the free algebra $\k\langle g,x,y\rangle$, where $x,y \in H$ are liftings of $a,b \in \grH$, respectively, subject to the relations 
\[
g^p=1,\ gx-xg=r_1,\ gy-yg=r_2,\ x^p=r_3,\ y^p=r_4,\ xy-yx=r_5,
\] 
for some $r_1,r_2\in H_0$, $r_5\in H_1$ and $r_3,r_4\in H_{p-1}$. The coalgebra structure is determined by  
\[\Delta(g)=g\otimes g,\  \Delta(x)=x\otimes 1+1\otimes x,\ \Delta(y)=y\otimes 1+1\otimes y.\]
Here, the primitive space is given by $P(H)=\k x\oplus \k y$. Similar computations as before, we have:
\begin{align*}
\Delta(r_1) &= r_1 \otimes g + g \otimes r_1; \text{ and} \\
\Delta(r_2) &= r_2 \otimes g + g \otimes r_2.
\end{align*}
Since $r_1, r_2 \in H_0=\k C_p$, we have $r_1=\gamma_1(g-g)=0$ and $r_2=\gamma_2(g-g)=0$, for some $\gamma_1, \gamma_2 \in \k$. Moreover, 
\begin{align*}
\Delta(r_3) &= \Delta(x^p) = (x\otimes 1+1\otimes x)^p = r_3 \otimes 1 + 1 \otimes r_3 \\
\Delta(r_4) &= \Delta(y^p) = (y\otimes 1+1\otimes y)^p = r_4 \otimes 1 + 1 \otimes r_4 \\
\Delta(r_5) &= \Delta(xy-yx) = r_5 \otimes 1 + 1 \otimes r_5.
\end{align*}
Thus, $r_3, r_4$, and $r_5$ are primitive elements and the possible relations in $H$ are:
$$
g^p=1, \ gx=xg,\ gy=yg,\ x^p= \lambda x + \mu y ,\ y^p= \sigma x + \tau y, \ xy-yx=\alpha x + \beta y,
$$
for some $\lambda, \mu, \sigma, \tau, \alpha, \beta \in \k$. Again, by Diamond Lemma~\cite{Beg}, all the ambiguities of the relations are resolvable. In this case, the primitive space $P(H)=\k\, x\oplus\k\, y$ is indeed a two-dimensional restricted Lie algebra. We can apply the classifications in \cite[Theorem 7.4 (1)-(5)]{wang2012connected} and obtain the following five classes of $H$:

\begin{itemize}
 \item \textbf{(A2-a):} $\k \langle x, y, g \rangle / (g^p=1, \, x^p = 0, \, y^p = 0, \, [g,x]= [g,y] = 0, \, [x,y]=0)$
\vspace{0.4em}

 \item \textbf{(A2-b):} $\k \langle x, y, g \rangle / (g^p=1, \, x^p = x, \, y^p = 0, \, [g,x]= [g,y] = 0, \, [x,y]=0)$
\vspace{0.4em}

 \item \textbf{(A2-c):} $\k \langle x, y, g \rangle / (g^p=1, \, x^p = y, \, y^p = 0, \, [g,x]= [g,y] = 0, \, [x,y]=0)$
\vspace{0.4em}

 \item \textbf{(A2-d):} $\k \langle x, y, g \rangle / (g^p=1, \, x^p = x, \, y^p = y, \, [g,x]= [g,y] = 0, \, [x,y]=0)$
\vspace{0.4em}

 \item \textbf{(A2-e):} $\k \langle x, y, g \rangle / (g^p=1, \, x^p = x, \, y^p = 0, \, [g,x]= [g,y] = 0, \, [x,y]=y)$,
\end{itemize}
with coalgebra structures $\Delta(g) = g \otimes g, \, \Delta(x) = x \otimes 1 + 1 \otimes x, \, \Delta(y) = y \otimes 1 + 1 \otimes y$ occur in all five classes. \\

\noindent 
\textbf{Case (A3), $p=2$.} It is clear that $H$ is the quotient of the free algebra $\k\langle g,x,y\rangle$, where $x,y \in H$ are liftings of $a,b \in \grH$, respectively, subject to the relations 
\[
g^p=1,\ gx-xg=r_1,\ gy-yg=r_2,\ x^p=r_3,\ y^p=r_4,\ xy-yx=r_5,
\] 
for some $r_1,r_2\in H_0$, $r_5\in H_1$ and $r_3,r_4\in H_{p-1}$. The coalgebra structure is determined by  
\[\Delta(g)=g\otimes g,\  \Delta(x)=x\otimes 1+1\otimes x,\ \Delta(y)=y\otimes 1+1\otimes y.\]
The computation is very similar to case (A2) and yields the same ambiguity relations. By applying \cite[Theorem 7.4 (1)-(5)]{wang2012connected} again, we also get five classes of $H$:

\begin{itemize}
 \item \textbf{(A3-a):} $\k \langle x, y, g \rangle / (g^p=1, \, x^p = 0, \, y^p = 0, \, gx=yg, \, gy=xg, \, [x,y]=0)$
\vspace{0.4em}

 \item \textbf{(A3-b):} $\k \langle x, y, g \rangle / (g^p=1, \, x^p = x, \, y^p = 0, \, gx=yg, \, gy=xg, \, [x,y]=0)$
\vspace{0.4em}

 \item \textbf{(A3-c):} $\k \langle x, y, g \rangle / (g^p=1, \, x^p = y, \, y^p = 0, \, gx=yg, \, gy=xg, \, [x,y]=0)$ 
\vspace{0.4em}

 \item \textbf{(A3-d):} $\k \langle x, y, g \rangle / (g^p=1, \, x^p = x, \, y^p = y, \, gx=yg, \, gy=xg, \, [x,y]=0)$
\vspace{0.4em}

 \item \textbf{(A3-e):} $\k \langle x, y, g \rangle / (g^p=1, \, x^p = x, \, y^p = 0, \, gx=yg, \, gy=xg, \, [x,y]=y)$,
\end{itemize}
with coalgebra structures $\Delta(g) = g \otimes g, \, \Delta(x) = x \otimes 1 + 1 \otimes x, \, \Delta(y) = y \otimes 1 + 1 \otimes y$ occur in all five classes.\\


\subsection{Liftings for Case (B), $p>2$}
\label{liftingB}

It is clear that $H$ is the quotient of the free algebra $\k\langle g,x,y\rangle$, where $x,y \in H$ are liftings of $a,b \in \grH$, respectively, subject to the relations 
\[
g^p=1,\ gx-xg=r_1,\ gy-(x+y)g=r_2,\ x^p=r_3,\ y^p=r_4,\ xy-yx-\frac{1}{2}x^2=r_5,
\] 
for some $r_1,r_2\in H_0$, $r_5\in H_1$ and $r_3,r_4\in H_{p-1}$. The coalgebra structure is determined by  
\[\Delta(g)=g\otimes g,\  \Delta(x)=x\otimes 1+g\otimes x,\ \Delta(y)=y\otimes 1+g\otimes y.\]
Here, $P(H)=0$. One computes:
\begin{align*}
\Delta(r_1) &= r_1 \otimes g + g^2 \otimes r_1; \text{ and} \\
\Delta(r_2) &= r_2 \otimes g + g^2 \otimes r_2.
\end{align*}
Since $r_1, r_2 \in H_0=\k C_p$, by rescaling the variables, we may write $r_1=\epsilon_1(g-g^2)$ and $r_2=\mu(g-g^2)$, for $\epsilon_1\in \{0,1\}$ and $\mu \in \k$.

Next, by applying \Cref{pthcoproduct} and similar argument as in case (A1), we have 
$$\Delta(r_3-\epsilon_1\, x)=(x^p-\epsilon_1\, x)\otimes 1+1\otimes (x^p-\epsilon_1\, x).$$
Hence, $(r_3-\epsilon_1\, x)$ is primitive. This implies $r_3 = \epsilon_1\, x$ since $P(H)=0$.
\begin{align*}
\Delta(r_5) &= \Delta(xy-yx-\frac{1}{2}x^2) \\
&= r_5 \otimes 1 + g^2 \otimes r_5 + (xg-gx) \otimes y + (gy - yg - \frac{1}{2}xg - \frac{1}{2}gx) \otimes x \\
&= r_5 \otimes 1 + g^2 \otimes r_5 - \epsilon_1(g-g^2) \otimes y + \mu (g-g^2) \otimes x - \frac{1}{2} \epsilon_1(g-g^2) \otimes x.
\end{align*}
One can check that 
\begin{align*}
\Delta(r_5+\epsilon_1\,y - \mu x + \frac{1}{2}\epsilon_1x) 
&= (r_5+\epsilon_1\,y - \mu x + \frac{1}{2}\epsilon_1x) \otimes 1 + g^2 \otimes (r_5+\epsilon_1\,y - \mu x + \frac{1}{2}\epsilon_1x) 
\end{align*}
So $(r_5+\epsilon_1y - \mu x + \frac{1}{2}\epsilon_1x)$ is $(1,g^2)$-skew primitive. One can write $r_5+\epsilon_1\,y -\mu x + \frac{1}{2}\epsilon_1x = \tau(1-g^2)$, for some $\tau \in \k$. In other words, $xy-yx=\frac{1}{2}x^2 - \epsilon_1\,y + (\mu - \frac{1}{2}\epsilon_1)x+\tau(1-g^2)$.

Case (B) becomes very complicated when it comes to computing the relation $r_4=y^p$, due to how $(g)(\ad \, y)=[g,y]=gy-yg$ involves both $x$ and $g$ terms as we have seen in $r_2$ computation above. Therefore, we specify it to the smallest case when $p=3$ and leave the general $p$ case for future task: \\

\noindent
\textbf{Case (B) for $p=3$:} We have the following relations 
\[
g^3=1,\ gx-xg=r_1,\ gy-(x+y)g=r_2,\ x^3=r_3,\ y^3=r_4,\ xy-yx-\frac{1}{2}x^2=r_5,
\] 
with the coalgebra structure 
\[\Delta(g)=g\otimes g,\  \Delta(x)=x\otimes 1+g\otimes x,\ \Delta(y)=y\otimes 1+g\otimes y.\]
As before, we can write $r_1=\epsilon(g-g^2), \ r_2=\mu(g-g^2), \ r_3=\epsilon\, x$ and $r_5= (\mu-\frac{1}{2}\epsilon)x-\epsilon\, y+\tau(1-g^2)$, for some $\epsilon \in \{0,1\}$ and $\mu, \tau \in \k$. For $r_4$, using the relations 
$$gy =yg+xg+ \mu(g-g^2), \ \ gx = xg+ \epsilon(g-g^2), \ \ xy=yx+ \frac{1}{2}x^2 + (\mu-\frac{1}{2}\epsilon)x - \epsilon y + \tau(1-g^2),$$
we compute in mod $3$
\begin{align*}
\Delta(y^2)=y^2\otimes 1+(2yg+xy+\mu g-\mu g^2)\otimes y+g^2\otimes y^2,
\end{align*}
and 
\begin{align*}
\Delta(y^3) &= (y \otimes 1 + g \otimes y) \left(y^2 \otimes 1 + (2yg+xg+\mu(g-g^2)) \otimes y + g^2 \otimes y^2 \right) \\ 
&= y^3 \otimes 1 + (2y^2g + yxg + \mu y(g-g^2) + gy^2) \otimes y \\
& \qquad + (yg^2 + 2gyg + gxg + \mu(g^2-1)) \otimes y^2 + g^3 \otimes y^3 \\
&=  y^3 \otimes 1 + (\epsilon xg - \epsilon yg - \mu\epsilon g^2 + (\tau + \mu^2)g + (\mu\epsilon - \tau - \mu^2)) \otimes y \\
& \qquad + \epsilon(g^2-1) \otimes y^2 + 1 \otimes y^3.
\end{align*}
Therefore,
\begin{align*}
\Delta(y^3-\epsilon\, y^2+(\mu\epsilon-\tau-\mu^2)\, y)=&\, (y^3-\epsilon\, y^2 +(\mu\epsilon -\tau-\mu^2)\, y)\otimes 1\\
&\, +1\otimes (y^3-\epsilon\, y^2+(\mu\epsilon -\tau-\mu^2)\,y). 
\end{align*}
So we have $(y^3-\epsilon\, y^2+(\mu\epsilon-\tau-\mu^2)\, y)$ is primitive, which is $0$, so $y^3 = r_4=\epsilon\, y^2-(\mu\epsilon-\tau-\mu^2)\, y$. The relations in $H$ (when $p=3$) are
\begin{align*}
g^3&=1,\ gx-xg=\epsilon(g-g^2),\ gy-yg=xg+\mu(g-g^2),\ x^3=\epsilon\, x,\\ 
y^3&=\epsilon\, y^2-(\mu\epsilon-\tau-\mu^2)\, y,\ xy-yx= -x^2 +(\epsilon+\mu)\,x-\epsilon\, y+\tau(1-g^2)
\end{align*}
for some $\epsilon \in \{0,1\}$ and $\mu,\tau \in \k$. When we apply the Diamond Lemma \cite{Beg}, all the ambiguities of the relations are resolvable and there is no ambiguity condition in this case. For instance, we have
\begin{align*}
[g,x]&\, =\epsilon(g-g^2)\\
[g,x^2]&\,=-\epsilon xg+\epsilon xg^2+\epsilon(g-1)\\
[g^2,x]&\,=\epsilon(1-g^2)\\
[g,y]&\,=xg+\mu(g-g^2)\\
[g^2,y]&\,=-xg^2+(\epsilon-\mu)(g^2-1)\\
[g,y^2]&\,=-yxg+\epsilon xg-(\mu+\epsilon)yg+\mu yg^2+(\mu\epsilon-\mu^2-\tau)+(\tau+\mu^2)g-\mu\epsilon g^2\\
[x^2,y]&\,=-(\tau+\mu\epsilon)x-\mu x^2+\tau xg^2+\epsilon yx+\epsilon y+\epsilon\tau(1-g^2)\\
[x,y^2]&\, =yx^2+(\epsilon-\mu)yx+\epsilon y^2+(\epsilon-\tau-\epsilon\mu)y+(\tau+\epsilon+\mu^2)x-\epsilon x^2+\tau yg^2+\tau \epsilon-\tau\epsilon g^2.
\end{align*}
Hence,
\begin{align*}
[g,y^3]&\, =[g,\epsilon y^2-(\mu\epsilon-\tau-\mu^2) y]\\
&\, =-\epsilon yxg+(\tau+\mu^2-\mu\epsilon+\epsilon)xg-(\epsilon\mu+\epsilon)yg+\mu\epsilon yg^2+(\mu \epsilon-\epsilon \tau-\mu^2\epsilon)\\
&\quad +(\epsilon \tau+\tau\mu+\mu^3)g+(\mu^2\epsilon-\mu\epsilon-\tau\mu-\mu^3)g^2\\
&\, =[[[g,y],y],y],
\end{align*}
and
\begin{align*}
[x,y^3]&\, =[g,\epsilon y^2-(\mu\epsilon-\tau-\mu^2) y]\\
&\, =\epsilon yx^2+(\epsilon-\epsilon\mu)yx+\epsilon y^2+(\epsilon+\tau\epsilon-\mu^2\epsilon)y+(\epsilon\mu^2+\tau\mu+\epsilon+\mu^3-\tau\epsilon-\epsilon\mu)x\\
&\ +(\mu\epsilon-\tau-\mu^2-\epsilon)x^2+\tau\epsilon yg^2+(\tau\epsilon+\tau^2+\tau\mu^2-\tau\mu\epsilon)(1-g^2)\\
&\, =[[[x,y],y],y].
\end{align*}
We may interpret these conditions further: 
\begin{itemize}
\item When $\epsilon=\mu=0$, by rescaling of both $x$ and $y$, we can further let $\tau\in \{0,1\}$. There are two classes of $H$.
\item When $\epsilon=0$ and $\mu\neq 0$, by rescaling of $x$ and $y$ by the same factor, we can further let $\mu=1$. There is one infinite family of $H$ depending on $\tau \in \k$.
\item When $\epsilon=1$, there is one infinite family of $H$ depending on two parameters $\mu,\tau \in \k$. 
\end{itemize}
In positive characteristic $p>3$, we conjecture that the lifting of case (B) is 
\begin{align*}
g^p&=1,\ gx-xg=\epsilon(g-g^2),\ gy-yg=xg+\mu(g-g^2),\ x^p=\epsilon\, x,\\ 
y^p&=f_{p-1}y^{p-1}+\cdots +f_1y,\ xy-yx= \frac{1}{2}x^2 +(\mu-\frac{1}{2}\epsilon)\,x-\epsilon\, y+\tau(1-g^2),
\end{align*}
where $\epsilon \in \{0,1\}$ and $\mu, \tau \in \k$ and $f_i$'s are polynomials in terms of $\epsilon,\mu,\tau$. Moreover, the isomorphism classes are given similarly as in the case $p=3$ above. \\


\subsection{Liftings for Case (C)}
\label{liftingC}

It is clear that $H$ is the quotient of the free algebra $\k\langle g,x,y\rangle$, where $x,y \in H$ are liftings of $a,b \in \grH$, respectively, subject to the relations 
\[
g^p=1,\ gx-xg=r_1,\ gy-yg=r_2,\ x^p=r_3,\ y^p=r_4,\ xy-yx=r_5,
\] 
for some $r_1\in H_0$, $r_2,r_3\in H_{p-1}$, $r_5\in H_p$ and $r_4\in H_{p^2-1}$. The coalgebra structure is determined by  
\[\Delta(g)=g\otimes g,\  \Delta(x)=x\otimes 1+g^{\epsilon}\otimes x,\ \Delta(y)=y\otimes 1+1\otimes y+\sum_{1\le i\le p-1}\frac{(p-1)!}{i!(p-i)!}\, x^i g^{\epsilon(p-i)} \otimes x^{p-i},\]
where $g^\epsilon=1$ or $g$. In order to show that the lifting of the comultiplication of $y$ is unique, we need the following result. \\

\noindent
\textbf{Claim:} Let $A$ be the Hopf subalgebra of $H$ generated by $\{g,x\}$ with relations
$$A=\k\langle g,x\rangle/(g^p-1,\, x^p-\epsilon x, \, gx-xg - \epsilon(g-g^2)),$$ 
for $\epsilon \in \{0,1\}$. Then $\dim \HL^2(\Omega A)=1$ and it is spanned by the 2-cocycle
\[
\sum_{1\le i\le p-1}\frac{(p-1)!}{i!(p-i)!}\, x^i g^{\epsilon(p-i)} \otimes x^{p-i}.
\]
\begin{proof}[Proof of claim]
Suppose $\epsilon=1$. By the classification result of pointed Hopf algebras of dimension $p^2$ over $\k$ \cite{WangWang}, $A$ is the only class that is both non-commutative and non-cocommutative. Hence it must be self dual and $A^*\cong A$ as Hopf algebras. 

Suppose $\epsilon=0$. Then $A$ is local and $\dim J/J^2=2$, where $J$ is the augmentation ideal of $A$. So $A^*\cong u(\mathfrak g)$ for some two-dimensional restricted Lie algebra $\mathfrak g$. Note that $A$ is non-cocommutative. So $\mathfrak g$ is non-abelian. By \cite[Theorem 7.4]{wang2012connected}, $\mathfrak g=\k x+\k y$ satisfying $x^p=x,y^p=0$ and $[x,y]=y$. 

Then it is clear to see that, as algebras, $A^*\cong \k C_p\# K$, where $K=\k[x]/(x^p-x)$ with $\Delta(x)=x\otimes 1+1\otimes x$ and $C_p$ is the cyclic group of order $p$ generated by $g$. The $K$-action on $\k C_p$ is given by $x\cdot g=g^2-g$ if $\epsilon=1$, or $x\cdot g=g+1$ if $\epsilon=0$. Now we can apply the spectral sequence used in Proposition~\ref{Hsmash} to conclude that 
$$\dim \HL^2(\Omega A)=\dim \HH^2(A^*,\k)=\dim \HH^2(\k C_p,\k)^K\le \dim \HH^2(\k C_p)=1.$$
Next one checks that the element $\sum_{1\le i\le p-1}\frac{(p-1)!}{i!(p-i)!}\, x^i g^{\epsilon(p-i)} \otimes x^{p-i}$ is a 2-cocycle in $\Omega A$ and does not lie in the coboundary. Hence, $\dim\HL^2(\Omega A)\ge 1$. So it must equal one. This proves the claim. 
\end{proof}

We obtain two cases depending on the value of $\epsilon$:
\begin{itemize}
 \item \textbf{(Ca)} Suppose $\epsilon=0$. Then $P(H)=\k x$, $\Delta(y)=y\otimes 1+1\otimes y+ \Bock(x)$, for the notation $\Bock(x)$ as in Equation (\ref{E:Bock}), and 
 $$\Delta(r_1)=\Delta(gx-xg) = r_1 \otimes g + g \otimes r_1.$$
 Since $r_1\in H_0$, we can write $r_1 = \gamma_1(g-g)=0$, for some $\gamma_1 \in \k$, that is, $gx=xg$. And
 $$\Delta(r_3)= \Delta(x^p) =r_3 \otimes 1 + 1 \otimes r_3.$$
So $r_3$ is primitive. By rescaling both $x$ and $y$, we may write $r_3 = \epsilon_3\, x$, for some $\epsilon_3 \in \{0,1\}$. 
\begin{align*}
\Delta(r_2)=\Delta(gy-yg) &= gy \otimes g + g \otimes gy + \sum_{1\le i\le p-1}\frac{(p-1)!}{i!(p-i)!} gx^i \otimes gx^{p-i} \\
& \qquad - yg \otimes g - g \otimes yg - \sum_{1\le i\le p-1}\frac{(p-1)!}{i!(p-i)!} x^ig \otimes x^{p-i}g \\
&= r_2 \otimes g + g \otimes r_2,
\end{align*}
since $gx^i=x^ig$, for any $i \geq 0$. Thus, we can write $r_2 = \gamma_2(g-g)=0$, for some $\gamma_2 \in \k$, that is, $gy=yg$. Similarly, one checks
$$\Delta(r_5)=\Delta(xy-yx) = 1 \otimes (xy-yx) + (xy-yx) \otimes 1.$$
So $r_5$ is primitive, we can write $r_5 = \sigma x$, for some $\sigma \in \k$.

For $r_4$, with all the relations we have in $H$, one sees that 
\[
\Bock(x) \left(\ad\ (\lambda\, \Bock(x)+y\otimes 1+1\otimes y )\right)^{p-1}=\Bock(x) \left(\ad\ (y\otimes 1+1\otimes y) \right)^{p-1}
\]
in $H\otimes H$. Then by Proposition~\ref{palgebra}, 
\begin{align*}
\Delta(r_4)&\, = \Delta(y^p)=y^p\otimes 1+1\otimes y^p+\omega (x)^p+\Bock(x) \left(\ad\ (y\otimes 1+1\otimes y) \right)^{p-1}\\
&\, =y^p\otimes 1+1\otimes y^p+\omega (x^p)+\rho_y^{p-1}\left(\Bock(x)\right),
\end{align*}
where we use notations in \cite{NWW2} to write $\rho_y(\Bock(x))=[\Bock(x),y\otimes 1+1\otimes y]$ and $\rho_y(x)=[x,y]$. We get two smaller cases depending on the characteristic of $\k$: \\

\begin{itemize}
 \item \textbf{(Ca')} Suppose char.$\k=2$. Then 
 \begin{align*}
 \rho_y(\Bock(x))&\, =[\Bock(x),y\otimes 1+1\otimes y]=[x\otimes x,y\otimes 1+1\otimes y]\\
 &\,=[x,y]\otimes x+x\otimes [x,y]=2\sigma\, x\otimes x=0.
 \end{align*}

Thus, one has:
 \begin{align*}
 \Delta(r_4) = y^p \otimes 1 + 1 \otimes y^p + \Bock(x^p)+\rho_y^{p-1}\left(\Bock(x)\right)= r_4 \otimes 1 + 1 \otimes r_4 + \epsilon_3 \, \Bock(x).
 \end{align*}
It is easy to check that $(r_4 - \epsilon_3 \,y)$ is primitive, so $r_4 - \epsilon_3 \,y = \tau x$, for some $\tau \in \k$, since $P(H)=\k x$. We have $r_4 = y^p = \tau x + \epsilon_3 \,y$. \\ 
 
For case (Ca') when $p=2$, the possible relations in $H$ are:
$$g^p=1,\ gx=xg,\ gy=yg,x^p=\epsilon_3 \,x, \ y^p=\tau x + \epsilon_3 \,y, \ xy-yx=\sigma x,$$
for some $\sigma, \tau \in \k$ and $\epsilon_3 \in \{0,1\}$. By Diamond Lemma~\cite{Beg}, we get the following compatible conditions for the coefficients:
$$\epsilon_3 \sigma = (\epsilon_3 - \sigma^{p-1})\sigma = \tau \sigma = 0.$$
Here if $\epsilon_3\sigma=0$, then we will have $\sigma^p=0$ which implies that $\sigma=0$ and it is easy to see then $\sigma=0$ is the only restriction. \\

\item \textbf{(Ca'')} Suppose char.$\k \neq 2$. Thus, by applying formula in \cite[Proposition 2.7(iii)]{NWW2} where $i=p-2$, we have:
$$\rho_y^{p-1}(\Bock(x))= \partial^1(Z) = \Delta(Z)- Z \otimes 1 - 1 \otimes Z,$$
where 
$$Z:= - \sum_{i_1+\ldots+i_p=p-2} \frac{(p-2)!}{i_1! \cdots i_p!} \rho_y^{i_1}(x) \cdots \rho_y^{i_{p-1}}(x) \rho_y^{1+i_p}(x).$$ 
By previous assumption,  $\rho_y(x)=[x,y] = \sigma x$, for some $\sigma\in \k$. So for any $s \geq 0,\ \rho_y^s(x)= \sigma^s x$. When $p>2$, $Z$ becomes:
\begin{align*}
Z&= - \sum_{i_1+\ldots+i_p=p-2} \frac{(p-2)!}{i_1! \cdots i_p!} \sigma^{i_1+\ldots+i_p+1} x^p \\
&= - \sum_{i_1+\ldots+i_p=p-2} \frac{(p-2)!}{i_1! \cdots i_p!} \sigma^{p-1} x^p \\
&= - \sum_{i_1+\ldots+i_p=p-2} \frac{(p-2)!}{i_1! \cdots i_p!} \sigma^{p-1} \epsilon_3 \,x \\ 
&= - \sigma^{p-1} \epsilon_3 \,x \left(\sum_{i_1+\ldots+i_p=p-2} \frac{(p-2)!}{i_1! \cdots i_p!} \right) \\
&= - \sigma^{p-1} \epsilon_3 \,x \, (p-2)^{p-2} = - \sigma^{p-1} \epsilon_3\, x \,(-2^{p-2}) \\
&= \sigma^{p-1} \epsilon_3 \,x \, \frac{1}{2} \, 2^{p-1} = \frac{1}{2} \sigma^{p-1} \epsilon_3\, x, 
\end{align*}
where $2^{p-1}=1$ by Fermat's Little Theorem. It follows that $\rho_y^{p-1}(\Bock(x))= \partial^1(Z) = \Delta(Z)- Z \otimes 1 - 1 \otimes Z$. So
$$ \Delta(r_4)=y^p \otimes 1 + 1 \otimes y^p + \epsilon_3 \, \Bock(x) + \Delta(Z)- Z \otimes 1 - 1 \otimes Z.$$
 Therefore, 
 \begin{align*}
 &\Delta(y^p - Z) = \Delta(r_4-\frac{1}{2} \sigma^{p-1} \epsilon_3\,x) \\
 &= (r_4-\frac{1}{2} \sigma^{p-1} \epsilon_3\,x) \otimes 1 + 1 \otimes (r_4-\frac{1}{2} \sigma^{p-1} \epsilon_3\,x) + \epsilon_3 \, \Bock(x).
 \end{align*} 
 It is easy to check that $(r_4-\frac{1}{2} \sigma^{p-1} \epsilon_3\,x - \epsilon_3 \, y)$ is primitive, so $r_4-\frac{1}{2} \sigma^{p-1} \epsilon_3\,x - \epsilon_3 \, y = \tau x$, for some $\tau \in \k$, or $y^p = (\tau + \frac{1}{2} \sigma^{p-1} \epsilon_3) x + \epsilon_3 \, y$. \\
 
For case (Ca'') where $p > 2$, the possible relations in $H$ are:
\begin{align*}
g^p&=1,\  gx=xg,\ gy=yg,\\ 
x^p=\epsilon_3 \,x,\ y^p&=(\tau + \frac{1}{2} \sigma^{p-1} \epsilon_3) x + \epsilon_3 \, y,\ xy-yx=\sigma x,
\end{align*}
for some $\sigma, \tau \in \k$ and $\epsilon_3 \in \{0,1\}$. By Diamond Lemma~\cite{Beg}, we get the following compatible conditions for the coefficients:
$$\epsilon_3 \sigma = (\epsilon_3 - \sigma^{p-1})\sigma = (\tau + \frac{1}{2} \sigma^{p-1} \epsilon_3) \sigma = 0.$$
Here if $\epsilon_3\sigma=0$, then we will have $\sigma^p=0$ which implies that $\sigma=0$ and it is easy to see then $\sigma=0$ is the only restriction. \\
\end{itemize}

Finally, the two cases (Ca') and (Ca") can be combined together. For case (Ca), the possible relations in $H$ are:
$$
g^p=1,\ gx=xg,\ gy=yg,x^p=\epsilon_3 \,x, \ y^p=\tau x + \epsilon_3 \,y, \ xy-yx=0,
$$
for some $\tau \in \k$ and $\epsilon_3 \in \{0,1\}$.
\begin{itemize}
\item Suppose $\epsilon_3=0$. By rescaling $x,y$, we can further choose $\tau\in \{0,1\}$. We get two finite classes of $H$. 
\item Suppose $\epsilon_3=1$. We get one infinite family of $H$ whose structures depending on $\tau \in \k$.\\
\end{itemize}

 \item \textbf{(Cb)} Suppose $\epsilon=1$. Then $P(H)=0$ and
 $$\Delta(r_1)=r_1 \otimes g + g^2 \otimes r_1.$$
 Since $r_1\in H_0$, we can write $r_1 = \epsilon_1(g-g^2)$, for $\epsilon_1 \in \{0,1\}$. 
 Again, by applying \Cref{pthcoproduct} and similar argument as in case (A1), we have $(r_3-\epsilon_1\, x)$ is primitive. This implies $r_3 = \epsilon_1\, x$ since $P(H)=0$.
\end{itemize}

The computations for $r_2, r_4, r_5$ in case (Cb) are very complicated due to the coproduct formula in $\Delta(y)$ and calculation in mod $p$. Here, we specify to case $(p=2)$ and case $(p > 2$ with additional assumption $gx=xg$). \\

\noindent
\textbf{Case (Cb) for $p=2$:} As before we have relations:
\[
g^2=1,\  gx-xg=\epsilon_1(g-g^2)=\epsilon_1(g-1),\ x^2=\epsilon_1x,
\]
for $\epsilon_1\in\{0,1\}$; and working in mod $2$ 
\begin{align*}
\Delta(gy-yg)=& (gy-yg)\otimes g+g\otimes (gy-yg)+\epsilon_1(x\otimes(g+1)+(g+1)\otimes xg) \\
& \,+\epsilon_1(1+g)\otimes (1+g).
\end{align*}
One can check that $gy-yg - \epsilon_1(1+g+x+xg)$ is $(g,g)$-primitive. This implies that $gy-yg=\epsilon_1(1+g+x+xg)$. Also,
\begin{align*}
\Delta(xy-yx)=(xy-yx)\otimes 1+g\otimes (xy-yx).
\end{align*}
So $xy-yx=\sigma x+\tau(1-g)$, for some $\sigma, \tau \in \k$, and 
\begin{align*}
\Delta(y^2)=y^2\otimes 1+1\otimes y^2+\tau\left((1-g)\otimes x+xg\otimes (1-g)\right)+\epsilon_1xg\otimes x.
\end{align*}
So $y^2=\epsilon_1y+\tau(x+xg)$. By Diamond Lemma~\cite{Beg}, we get all compatible conditions as: $\sigma \tau=\sigma^2-\epsilon_1\sigma=\epsilon_1\sigma=0$. So we have $\sigma=0$ and the relations are 
\begin{gather*}
g^2=1,\  gx-xg=\epsilon_1(g-1),\ gy-yg=\epsilon_1(1+g+x+xg),\\
xy-yx=\tau(1-g),\ x^2=\epsilon_1x,\ y^2=\epsilon_1y+\tau (x+xg).
\end{gather*}

If $\epsilon_1=0$, by rescaling $x,y$, we can choose $\tau\in \{0,1\}$, where we get two classes of $H$. If $\epsilon_1=1$, we get one infinite family of $H$ whose structures depending on $\tau\in \k$. \\

\noindent
\textbf{Case (Cb) for $p>2$ with additional assumption:} Suppose $[g,x]=0$. We can get that $x^p=0$ and  
$$\Delta(gy-yg)=(gy-yg)\otimes g+g\otimes (gy-yg).$$
This implies that $gy=yg$. And
$$\Delta(xy-yx)=(xy-yx)\otimes 1+g\otimes (xy-yx).$$
So $(xy-yx)$ is $(1,g)$-primitive. By rescaling of $x$, which is also $(1,g)$-primitive, we can write $xy-yx=\epsilon_2\,x+\tau(1-g)$, for some $\tau\in \k$ and $\epsilon_2 \in \{0,1\}$. 

Let $A$ be the commutative Hopf subalgebra of $H$ generated by $\{g,x\}$. We follow the notations in \cite{NWW2} to denote, depending on the context, $\rho_y=(\ad\, y)$ as the right adjoint action of $y$ on $A$, or $\rho_y=(\ad\, (y\otimes 1+1\otimes y))$ as the right adjoint action of $(y\otimes 1+1\otimes y)$ on $A\otimes A$. Then by \Cref{palgebra}, one sees that 
\begin{align*}
\Delta(y^p)=&\, \left(y\otimes 1+1\otimes y+\sum_{1\le i\le p-1}\frac{(p-1)!}{i!(p-i)!}\, x^i g^{p-i} \otimes x^{p-i}\right)^p\\
=&\ y^p\otimes 1+1\otimes y^p+\rho_y^{p-1}\left(\sum_{1\le i\le p-1}\frac{(p-1)!}{i!(p-i)!}\, x^i g^{p-i} \otimes x^{p-i}\right).
\end{align*}

We apply the fact that $\rho_y(x)=[x,y]=\epsilon_2\,x+\tau(1-g)$.
\begin{align*}
&\,\rho_y\left(\sum_{1\le i\le p-1}\frac{(p-1)!}{i!(p-i)!} x^i g^{p-i} \otimes x^{p-i}\right)\\
=&\,\sum_{1\le i\le p-1}\frac{(p-1)!}{i!(p-i)!} \rho_y(x^i) g^{p-i} \otimes x^{p-i}+\sum_{1\le i\le p-1}\frac{(p-1)!}{i!(p-i)!} x^i g^{p-i} \otimes \rho_y(x^{p-i})\\
=&\, \epsilon_2i\, \sum_{1\le i\le p-1}\frac{(p-1)!}{i!(p-i)!} x^ig^{p-i} \otimes x^{p-i}+\tau i\, \sum_{1\le i\le p-1}\frac{(p-1)!}{i!(p-i)!} x^{i-1}(1-g)g^{p-i} \otimes x^{p-i}\\
+ \epsilon_2&(p-i) \sum_{1\le i\le p-1}\frac{(p-1)!}{i!(p-i)!} x^ig^{p-i} \otimes x^{p-i} +\tau (p-i) \sum_{1\le i\le p-1}\frac{(p-1)!}{i!(p-i)!} x^ig^{p-i} \otimes x^{p-i-1}(1-g)\\
=&\, \tau\left(\sum_{0\le j\le p-2}{p-1 \choose j} g^{p-1-j}x^j\otimes x^{p-1-j}-\sum_{0\le j\le p-2} {p-1\choose j} g^{p-j} x^j\otimes x^{p-1-j}\right.\\
&\quad +\left. \sum_{1\le j\le p-1} {p-1\choose j} g^{p-j}x^j\otimes x^{p-1-j}-\sum_{1\le j\le p-1}{p-1 \choose j} g^{p-j}x^j\otimes gx^{p-1-j}\right)\\
=&\, \tau \left([\Delta(x^{p-1})-x^{p-1}\otimes 1] - g^p\otimes x^{p-1} + gx^{p-1}\otimes 1+ [-\Delta(gx^{p-1})+g^p\otimes gx^{p-1}] \right)\\
=&\, \tau \left(\Delta(x^{p-1}(1-g))-(x^{p-1}(1-g))\otimes 1-1\otimes (x^{p-1}(1-g)) \right)\\
=&\, \tau \partial^1(x^{p-1}(1-g)),
\end{align*}
where $\partial^1: A\to A\otimes A$ such that $\partial^1(a)=\Delta(a)-a\otimes 1-1\otimes a$, for any $a \in A$. Next, we need the following result. \\

\noindent
\textbf{Claim:} $\rho_y$ and $\partial^1$ commute with each other. 
\begin{proof}[Proof of claim]
For any $a\in A$, we get
\begin{align*}
\rho_y\partial^1(a)&\, =\rho_y(\Delta(a)-a\otimes 1-1\otimes a)\\
&\, =[\Delta(a)-a\otimes 1-1\otimes a,y\otimes 1+1\otimes y]\\
&\, =[\Delta(a),y\otimes 1+1\otimes y]-[a,y]\otimes 1-1\otimes [a,y]\\
&\, =[\Delta(a),y\otimes 1+1\otimes y+\sum_{1\le i\le p-1} \frac{(p-1)!}{i!(p-i)!}\, x^i g^{p-i} \otimes x^{p-i}]-[a,y]\otimes 1-1\otimes [a,y]\\
&\, =[\Delta(a),\Delta(y)]-[a,y]\otimes 1-1\otimes [a,y]\\
&\, =\Delta([a,y])-[a,y]\otimes 1-1\otimes [a,y]\\
&\, =\partial^1([a,y]) =\partial^1(\rho_y(a)).
\end{align*}
\end{proof}
Therefore, one sees that 
\begin{align*}
\rho_y^{p-1}\left(\sum_{1\le i\le p-1}\frac{(p-1)!}{i!(p-i)!} x^i g^{p-i} \otimes x^{p-i}\right)&=\rho_y^{p-2}\left(\tau \,\partial^1(x^{p-1}(1-g))\right) \\
&= \tau \, \partial^1((1-g)\rho_y^{p-2}(x^{p-1})),
\end{align*} so
\[
\Delta(y^p- \tau \, \rho_y^{p-2}(x^{p-1})(1-g))=(y^p- \tau \, \rho_y^{p-2}(x^{p-1})(1-g))\otimes 1+1\otimes (y^p- \tau \, \rho_y^{p-2}(x^{p-1})(1-g)).
\]
Since $P(H)=0$, we get $y^p=\tau \, (1-g)\rho_y^{p-2}(x^{p-1})$. Since the ambiguity $[x,y^p]=0$ is resolvable, we conclude that $0=[x, \tau \, \rho_y^{p-2}(x^{p-1})(1-g)]=x(\ad\ y)^p=\epsilon_2^p x+\epsilon_2^{p-1}\tau(1-g)$. Hence, $\epsilon_2=0$ and $\rho_y(x)=\tau(1-g)$. Moreover, we get 
\[
y^p=\tau \,(1-g)\rho_y^{p-2}(x^{p-1})=(p-1)!(1-g)\tau^{p-1}(1-g)^{p-2}x=-\tau^{p-1}(1-g)^{p-1}x.
\]
Thus the relations are 
\[
g^p=1,\  gx-xg=0,\ gy-yg=0,\ xy-yx=\tau(1-g),\ x^p=0,\ y^p=-\tau^{p-1}(1-g)^{p-1}x.
\]
When we apply the Diamond Lemma \cite{Beg}, all the ambiguities of the relations are resolvable and there is no ambiguity condition in this case. By rescaling of $x$ and $y$, we can take $\tau \in \{0,1\}$. There are two classes of $H$.
\\


\subsection{Liftings for Case (D)} \
\label{liftingD}

\noindent 
\textbf{Case (D1).} It is clear that $H$ is the quotient of the free algebra $\k\langle g,x\rangle$, where $x \in H$ is lifting of $a \in \grH$, subject to the relations 
\[
g^{p^2}=1,\ gx-xg=r_1,\  x^p=r_2,
\] 
for some $r_1\in H_0$ and $r_2\in H_{p-1}$. The coalgebra structure is determined by  
\[\Delta(g)=g\otimes g,\  \Delta(x)=x\otimes 1+g^{\epsilon}\otimes x,\]
where $g^\epsilon=1,g$, or $g^p$. We compute 
$$\Delta(r_1) = r_1 \otimes g + g^{\epsilon + 1} \otimes r_1.$$
Since $r_1\in H_0$, we can write $r_1 = \epsilon_1(g-g^{\epsilon+1})$, for $\epsilon_1 \in \{0,1\}$. 
$$\Delta(r_2) = (x \otimes 1 + g^\epsilon \otimes x)^p.$$

We have the following three cases depending on the values of $\epsilon$:
\begin{itemize}
 \item \textbf{(D1a)} Suppose $\epsilon = 0$. Then $P(H)=\k x$, $r_1 = 0$, and $r_2 = \lambda x$, for some $\lambda \in \k$. Thus, after checking all ambiguities and by Diamond Lemma, the relations in $H$ are: 
 $$gx=xg,\  x^p=\lambda x.$$
No ambiguity conditions occur for case (D1a). By rescaling of $x$, we can take $\lambda =\epsilon_2 \in \{0,1\}$. We obtain two classes of $H$ whose structures depending on $\epsilon_2$. \\
 
 \item \textbf{(D1b)} Suppose $\epsilon = 1$. Then by applying \Cref{pthcoproduct}, we have $\Delta(x^p-\epsilon_1x)=(x^p-\epsilon_1x)\otimes 1+g^p\otimes (x^p-\epsilon_1x)$. Hence, $x^p-\epsilon_1x = \lambda(1-g^p)$, for some $\lambda \in \k$, and $r_2=x^p=\epsilon_1 x+\lambda(1-g^p)$. The possible relations in $H$ are:
 $$gx-xg=\epsilon_1(g-g^2),\  x^p=\epsilon_1 x+\lambda(1-g^p),$$
for $\epsilon_1 \in \{0,1\}$ and some $\lambda \in \k$. By Diamond Lemma~\cite{Beg}, all the ambiguities of the relations are resolvable. There is no ambiguity condition in this case. If $\epsilon_1=0$, by rescaling $x$, we can choose $\lambda\in\{0,1\}$. There are two classes. Otherwise, when $\epsilon_1=1$, we obtain one infinite family of $H$ whose structures depending on $\lambda \in \k$. \\

 \item \textbf{(D1c)} Suppose $\epsilon = p$. Observe that since $[-,x]$ is a derivation on $\k\langle g\rangle$, $[g^p,x]=pg^{p-1}[g,x]=0$ in characteristic $p$, so $g^px = xg^p$. We have
 $$\Delta(r_2) = (x \otimes 1 + g^p \otimes x)^p = x^p \otimes 1 + (g^p)^p \otimes x^p = x^p \otimes 1 + 1 \otimes x^p,$$
since $g^{p^2} = 1$. Thus, $r_2 \in P(H)=0$ implying $x^p = 0$. The relations in $H$ are: 
$$gx-xg=\epsilon_1(g-g^{p+1}),\  x^p=0,$$
 for $\epsilon_1 \in \{0,1\}$.  
%
%
By Diamond Lemma~\cite{Beg}, all the ambiguities of the relations are resolvable. There is no ambiguity condition in this case. We make a remark that for $[g,x^p]$ we need to apply the identity  $(g)(\ad \, x)^n= \epsilon_1^{n-1}(1-g^p)^{n-1}[g,x]$, for all $n \geq 1$, which can be proved inductively. Thus we obtain two classes of $H$ whose structures depending on $\epsilon_1 \in \{0,1\}$. \\
\end{itemize}

\noindent 
\textbf{Case (D2).} It is clear that $H$ is the quotient of the free algebra $\k\langle g_1,g_2,x\rangle$, where $x \in H$ is lifting of $a \in \grH$, subject to the relations 
\[
g_1^{p}=1,\ g_2^p=1, g_1x-xg_1=r_1,\  g_2x-xg_2=r_2,\ x^p=r_3,
\] 
for some $r_1,r_2\in H_0$ and $r_3\in H_{p-1}$. The coalgebra structure is determined by  
\[\Delta(g_i)=g_i\otimes g_i,\  \text{for } i=1,2; \ \Delta(x)=x\otimes 1+g_1^{\epsilon}\otimes x,\]
where $g_1^\epsilon=1$ or $g_1$. As before,
\begin{align*}
\Delta(r_1) &= r_1 \otimes g_1 + g_1^{\epsilon+1} \otimes r_1; \text{ and} \\
\Delta(r_2) &= r_2 \otimes g_2 + g_1^\epsilon g_2 \otimes r_2.
\end{align*}
Since $r_1, r_2 \in H_0$, and since we can rescale $x$ only once, we may write $r_1 = \epsilon_1(g_1-g_1^{\epsilon+1})$ and $r_2 = \tau(g_2-g_1^\epsilon g_2)$, for $\epsilon_1 \in \{0,1\}$ and some $\tau \in \k$. 
We have the following two cases depending on the value of $\epsilon$: 
\begin{itemize}
 \item \textbf{(D2a)} Suppose $\epsilon = 0$. Then $P(H)=\k x$, $r_1=0$, $r_2=0$, and $r_3 = \lambda x$, for some $\lambda \in \k$. By Diamond Lemma~\cite{Beg}, all the ambiguities of the relations are resolvable. The relations in $H$ are: 
 $$g_1x=xg_1,\  g_2x=xg_2, \ x^p=\lambda x,$$
for some $\lambda \in \k.$ By rescaling of $x$, $\lambda$ can be chosen as $\epsilon_2 \in \{0,1\}$. No ambiguity conditions occur for case (D2a). We obtain two classes of $H$ whose structures depending on $\epsilon_2$. \\
 
 \item \textbf{(D2b)} Suppose $\epsilon = 1$. Then by applying \Cref{pthcoproduct}, we have $(r_3-\epsilon_1 x)$ is primitive. This implies $r_3 = \epsilon_1 x$ since $P(H)=0$. The relations in $H$ are:
 $$g_1x-xg_1=\epsilon_1g_1(1-g_1),\ g_2x-xg_2= \tau g_2(1-g_1), \ x^p= \epsilon_1 x,$$
for $\epsilon_1\in \{0,1\}$ and some $\tau \in \k$. By Diamond Lemma~\cite{Beg}, all the ambiguities of the relations are resolvable. The ambiguity $[g_2,x^p]=(g_2)(\ad\, x)^p=(g_2)(\ad \, x)$ can be checked by using \Cref{pthpower} (3), where we take $\delta=(\ad \,x)$. Then one sees that $\delta(g_1)=\epsilon_1(g_1-g_1^2)$ and $\delta(g_2)=\tau g_2(1-g_1)$. If $\epsilon_1=0$, then one may rescale $x$ and assume $\tau = \epsilon_2 \in \{0,1\}$; and we get two classes of $H$ depending on $\epsilon_2$. If $\epsilon_1=1$, we get one infinite family of $H$ whose structures depending on $\tau \in \k$ and $\epsilon_1=1$. \\
\end{itemize}

\noindent
\textbf{Acknowledgements:} The main work of this paper results from the second author's visit to Northeastern University in March 2016. The second author is thankful for their hospitality.


\end{document}